\documentclass[english]{amsart}
\usepackage{mathptmx}

\usepackage[T1]{fontenc}
\usepackage[latin1]{inputenc}
\usepackage{amsthm}
\usepackage{graphicx}
\usepackage{bbm}
\usepackage{mathptmx}
\usepackage{setspace}
\usepackage{color}
\usepackage{cancel}
\usepackage{a4wide}
\makeatletter

\usepackage{babel}

\numberwithin{equation}{section} 
\numberwithin{figure}{section} 
\theoremstyle{plain}
\newtheorem{thm}{Theorem}
  \theoremstyle{plain}
  \newtheorem{cor}{Corollary}[section]
  \theoremstyle{plain}
  \newtheorem{lem}[cor]{Lemma}
  \theoremstyle{plain}
    \newtheorem{prop}[cor]{Proposition}
    \theoremstyle{remark}
       \newtheorem{rem}[cor]{Remark}
	 \theoremstyle{plain}

\theoremstyle{plain}
  
\usepackage{amssymb,amsmath}
\usepackage{bbm}

\DeclareMathOperator{\diam}{diam}

\renewcommand{\phi}{\varphi}

\renewcommand{\tilde}{\widetilde}

\def\diam {\mathop {\hbox{\rm diam}}}

\def\R{\mathbb{R}}

\def\N{\mathbb{N}}

\def\U{[0,1]}
\SetSymbolFont{operators}{bold}{OT1}{cmr}{bx}{n}
\SetSymbolFont{letters}{bold}{OML}{cmm}{b}{it}
\SetSymbolFont{symbols}{bold}{OMS}{cmsy}{b}{n}

\usepackage{babel}

\numberwithin{equation}{section} 
\numberwithin{figure}{section} 

  \theoremstyle{plain}

  \theoremstyle{plain}

\makeatother

\usepackage{babel}
  \addto\captionsbritish{}
  \addto\captionsbritish{}
  \addto\captionsbritish{}
  \addto\captionsbritish{}
  \addto\captionsenglish{}
  \addto\captionsenglish{}
  \addto\captionsenglish{}
  \addto\captionsenglish{}

\begin{document}

\title{On the derivative of the  $\alpha$-Farey-Minkowski function}

\author{Sara Munday}

\address{Fachbereich 3 - Mathematik und Informatik, Universität Bremen, Bibliothekstr.
1, D-28359 Bremen, Germany}

\email{smunday@math.uni-bremen.de}

\

\begin{abstract}
In this paper we  study the family of  $\alpha$-Farey-Minkowski functions $\theta_\alpha$, for an arbitrary countable partition $\alpha$
of the unit interval with atoms which accumulate only at the origin, which are the conjugating homeomorphisms between each of the $\alpha$-Farey systems and the tent map. We first show that each function $\theta_\alpha$ is singular with respect to the Lebesgue measure and then demonstrate that the unit interval can be written as the disjoint union of the following three sets: $\Theta_0:=\{x\in\U:\theta_\alpha'(x)=0\}, \ \Theta_\infty:=\{x\in\U:\theta_\alpha'(x)=\infty\}\ \text{ and } \Theta_\sim:=\U\setminus(\Theta_0\cup\Theta_\infty)$. The main result is that
\[
\dim_{\mathrm{H}}(\Theta_\infty)=\dim_{\mathrm{H}}(\Theta_\sim)=\sigma_\alpha(\log2)<\dim_{\mathrm{H}}(\Theta_0)=1,
\]
where $\sigma_\alpha(\log2)$ is the Hausdorff dimension of the level set $\{x\in \U:\Lambda(F_\alpha, x)=s\}$, where $\Lambda(F_\alpha, x)$ is the Lyapunov exponent of the map $F_\alpha$ at the point $x$. The proof of the theorem employs the multifractal formalism for $\alpha$-Farey systems.
\end{abstract}

\maketitle

\vspace{-10mm}

\section{Introduction and Statement of Results}

The aim of this paper is to study the family of  $\alpha$-Farey-Minkowski maps, which we denote by $\theta_\alpha$, where $\alpha:=\{A_n:n\in\N\}$ denotes a countable partition of the unit interval into  non-empty, right-closed and left-open intervals. These maps were first introduced in \cite{KMS}.  In that paper, the $\alpha$-Farey and $\alpha$-L\"uroth systems were also introduced and investigated. We will provide some details of these systems in Section 2, but let us simply mention now that for a given partition $\alpha$, the  $\alpha$-Farey-Minkowski map $\theta_\alpha$ is the conjugating homeomorphism between the $\alpha$-Farey map $F_\alpha$ and the tent map $T$. This means that $\theta_\alpha$ is a homeomorphism of the unit interval such that $\theta_\alpha\circ F_\alpha= T\circ \theta_\alpha$.

Our first result is that for every partition $\alpha$ (with the exception of the dyadic partition $\alpha_D$, which is defined by $\alpha_D:=\{(1/2^n, 1/2^{n-1}]:n\in\N\}$), if the derivative  $\theta_\alpha'(x)$ exists {\em in a generalised sense}, meaning that it either exists or we have that $\theta_\alpha'(x)=\infty$, then
\[
\theta_\alpha'(x)\in \{0, \infty\}.
\]
For the dyadic partition, since the map $F_{\alpha_D}$ can easily be seen to coincide with the tent map, the map $\theta_{\alpha_D}$ is nothing other than the identity map on $[0,1]$. We then show that from this it follows that for an arbitrary non-dyadic partition $\alpha$,
the map $\theta_\alpha$ is singular with respect to the Lebesgue measure $\lambda$.  In other words, we have that for $\lambda$-a.e. $x\in\U$, the derivative $\theta_\alpha'(x)$ exists and is equal to zero.
Consequently, the unit interval can be split into three pairwise disjoint sets $\Theta_0, \Theta_\infty$ and $\Theta_{\sim}$, which are defined as follows:
\[
\Theta_0:=\{x\in\U:\theta_\alpha'(x)=0\}, \ \Theta_\infty:=\{x\in\U:\theta_\alpha'(x)=\infty\}\ \text{ and } \Theta_\sim:=\U\setminus(\Theta_0\cup\Theta_\infty).
\]
It is immediate from the results stated above that
\[
\lambda(\Theta_0)=\dim_{\mathrm{H}}(\Theta_0)=1,
\]
where $\dim_{\mathrm{H}}(A)$ denotes the Hausdorff dimension of a  set $A\subseteq \R$.

For all the remaining results of the paper, we must restrict  the class of partitions to those that are either expanding or expansive of exponent $\tau\geq 0$ and eventually decreasing (the relevant definitions are given in Section 3 below). The first main result of the paper is concerned with relating the derivative of $\theta_\alpha$ to the sets $\mathcal{L}(s)$, which are defined as follows:
\[
\mathcal{L}(s):=\left\{x\in \U:\lim_{n\to\infty}\frac{\log(\lambda(I^{(\alpha)}_n(x)))}{-n}=s\right\},
\]
where $I_n^{(\alpha)}(x)$ refers to the unique $\alpha$-Farey cylinder set containing the point $x$  (see Section 2 for the precise definition).
These sets are only non-empty for $s$ inside the interval $[s_-, s_+]$, where $s_-:=\inf\{-\log(a_n)/n:n\in\N\}$ and $s_+:=\sup\{-\log(a_n)/n:n\in\N\}$. We obtain that if $s\in [s_-, \log2)$, then $\mathcal{L}(s)\subset \Theta_\infty$, whereas if $s\in(\log 2, s_+]$, then $\mathcal{L}(s)\subset \Theta_0$.
The significance of the $\log2$ is that this is the value of the topological entropy of each map $F_\alpha$.

The second main result of this paper is to employ the multifractal results obtained in \cite{KMS} to calculate the Hausdorff dimensions of the sets $\Theta_\infty$ and $\Theta_\sim$.
%
%
We have the following theorem.
\begin{thm}\label{mainthm}
\[
\dim_{\mathrm{H}}(\Theta_\infty)=\dim_{\mathrm{H}}(\Theta_\sim)=\mathcal{L}(\log2)<\dim_{\mathrm{H}}(\Theta_0)=1.
\]
\end{thm}
This theorem is proved by employing the results obtained for the Hausdorff dimension of the Lyapunov spectrum of $F_\alpha$ in \cite{KMS}, after first observing that the set $\mathcal{L}(s)$ coincides, up to a countable set of points, with the set $\{x\in [0,1]:\Lambda(F_\alpha, x)=s\}$, where $\Lambda(F_\alpha, x)$ refers to the Lyapunov exponent of the map $F_\alpha$ at the point $x$. All the necessary definitions and results are recalled at the start of Section 4.


\section{The $\alpha$-L\"uroth and  $\alpha$-Farey systems, and the function $\theta_\alpha$}

In this section, we wish to remind the reader of the definition and some basic properties of the $\alpha$-L\"uroth and $\alpha$-Farey systems, which were introduced in \cite{KMS} (let us also  mention that the $\alpha$-L\"uroth systems are a particular class of generalised L\"uroth system, as introduced in \cite{BBDK}).

Recall from the introduction that $\alpha:=\{A_n:n\in\N\}$ denotes a countably infinite partition of the unit interval $\U$, consisting of non-empty, right-closed and left-open intervals, and let $a_n:=\lambda(A_n)$ and $t_n:=\sum_{k=n}^\infty a_k$. It is assumed throughout that the elements of $\alpha$ are ordered from right to left, starting from $A_1$, and that these elements accumulate only at the origin. Then, for a given partition $\alpha$, the {\em $\alpha$-L\"uroth map} $L_\alpha:\U\to\U$ is defined to be
\[
L_{\alpha}(x):=
\left\{
	\begin{array}{ll}
	    ({t_n-x})/a_n & \text{ for }x\in A_n,\ n\in\N;\\
	  0 & \hbox{ if } x=0.
	\end{array}
      \right.
\]

Each map $L_\alpha$ allows us to obtain a representation of the numbers in $[0,1]$. We will refer to this expansion as the {\em $\alpha$-L\"uroth expansion}. As shown in \cite{KMS}, for each $x\in (0, 1]$, the finite or infinite sequence
$(\ell_k)_{k\geq1}$ of positive integers is determined by $L_{\alpha}^{k-1}(x) \in
A_{\ell_{k}}$, where the sequence terminates in $k$ if and only if  $L_{\alpha}^{k-1}(x)=t_n$, for some $n\geq2$.
Then the $\alpha$-L\"{u}roth expansion of  $x$ is
given as follows, where the sum is supposed to be finite if the
sequence is finite:
\[
x=
\sum_{n=1}^\infty(-1)^{n-1}\left(\textstyle\prod\limits_{i<n}a_{\ell_i}\right)
t_{\ell_n}=t_{\ell_1}-a_{\ell_1}t_{\ell_2}+a_{\ell_1}a_{\ell_2}t_{\ell_3}+\cdots.
\]
 In this situation we then write
$x=[ \ell_1, \ell_2,
\ell_3, \ldots]_{\alpha}$ for a point $x\in\U$ with an infinite $\alpha$-L\"uroth expansion and $x=[\ell_1, \ldots, \ell_k]_\alpha$ for a finite $\alpha$-L\"uroth expansion.
It is easy to see
that every infinite expansion is unique, whereas each
$x\in(0,1)$ with a finite $\alpha$-L\"uroth expansion can be expanded
in exactly two ways. Namely, one immediately verifies that $x=[\ell_1, \ldots, \ell_k,
1]_\alpha=[\ell_1,  \ldots, \ell_{k-1}, (\ell_k +1)]_\alpha$. By
analogy with continued fractions, for which a number is rational if
and only if it has a finite continued fraction expansion, we say that
$x\in\U$ is an \textit{$\alpha$-rational number} when $x$ has a
finite $\alpha$-L\"uroth expansion and say that $x$ is an \textit{$\alpha$-irrational
number} otherwise.

\vspace{1mm}

We will now define the cylinder sets associated with the map $L_\alpha$. For each $k$-tuple $(\ell_1, \ldots, \ell_k)$ of positive integers, define the {\em $\alpha$-L\"uroth cylinder set} $C_\alpha (\ell_1, \ldots, \ell_k)$ associated with  the $\alpha$-L\"{u}roth expansion
 to be
\[
C_\alpha (\ell_1, \ldots, \ell_k):=\{[ y_1, y_2, \ldots
]_\alpha:y_i=\ell_i\text{ for } 1\leq i\leq k\}.
\]
Observe that these sets are closed intervals with endpoints
given by $[\ell_1, \ldots, \ell_k]_\alpha$ and
$[\ell_1,
\ldots, (\ell_k +1)]_\alpha$. If $k$ is even, it follows  that $[\ell_1, \ldots, \ell_k]_\alpha$ is the left endpoint of this interval. Likewise, if $k$ is odd, $[\ell_1, \ldots, \ell_k]_\alpha$ is the right endpoint. For the Lebesgue measure of these sets we have that
\[
\lambda(C_\alpha(\ell_1, \ldots, \ell_k))=a_{\ell_1}\ldots a_{\ell_k}.
\]

\vspace{1mm}

Let us now recall some details of the $\alpha$-Farey map, $F_\alpha:\U\to\U$. For a given partition  $\alpha$, the map $F_{\alpha}:\U \to \U$ is given by
\[
F_{\alpha}(x):=\left\{
        \begin{array}{ll}
          (1-x)/a_1 & \hbox{if $x\in A_1$,} \\
          {a_{n-1}}(x-t_{n+1})/a_{n}+t_n & \hbox{if $x\in A_n$, for  $n\geq2$,}
\\
	  0 & \hbox{if $x=0$. }
        \end{array}
      \right.
\]

An example of the graph of  an $\alpha$-Farey and an $\alpha$-L\"uroth map is shown in Figure 2.1, for the specific example of the harmonic partition, $\alpha_H:=\{(1/(n+1), 1/n]:n\in\N\}$. For another specific example, consider the dyadic partition
$\alpha_D:=\left\{\left(1/2^{n},1/2^{n-1}\right]:n\in\N\right\}$. One can  immediately verify that
the map  $F_{\alpha_D}$ coincides with  the tent map $T:[0,1]\to[0,1]$, which is given by \[T(x):=\left\{
                                                                                                   \begin{array}{ll}
                                                                                                     2x, & \hbox{for $x\in[0,1/2)$;} \\
                                                                                                     2-2x, & \hbox{for $x\in [1/2, 1]$.}
                                                                                                   \end{array}
                                                                                                 \right.
\]
To see this,  it is enough to note that for each $n\in\N$ we have that $a_n=2^{-n}$ and $t_n=2^{-(n-1)}$.

\begin{figure}[ht!]
\begin{center}
\includegraphics[width=0.42\textwidth]{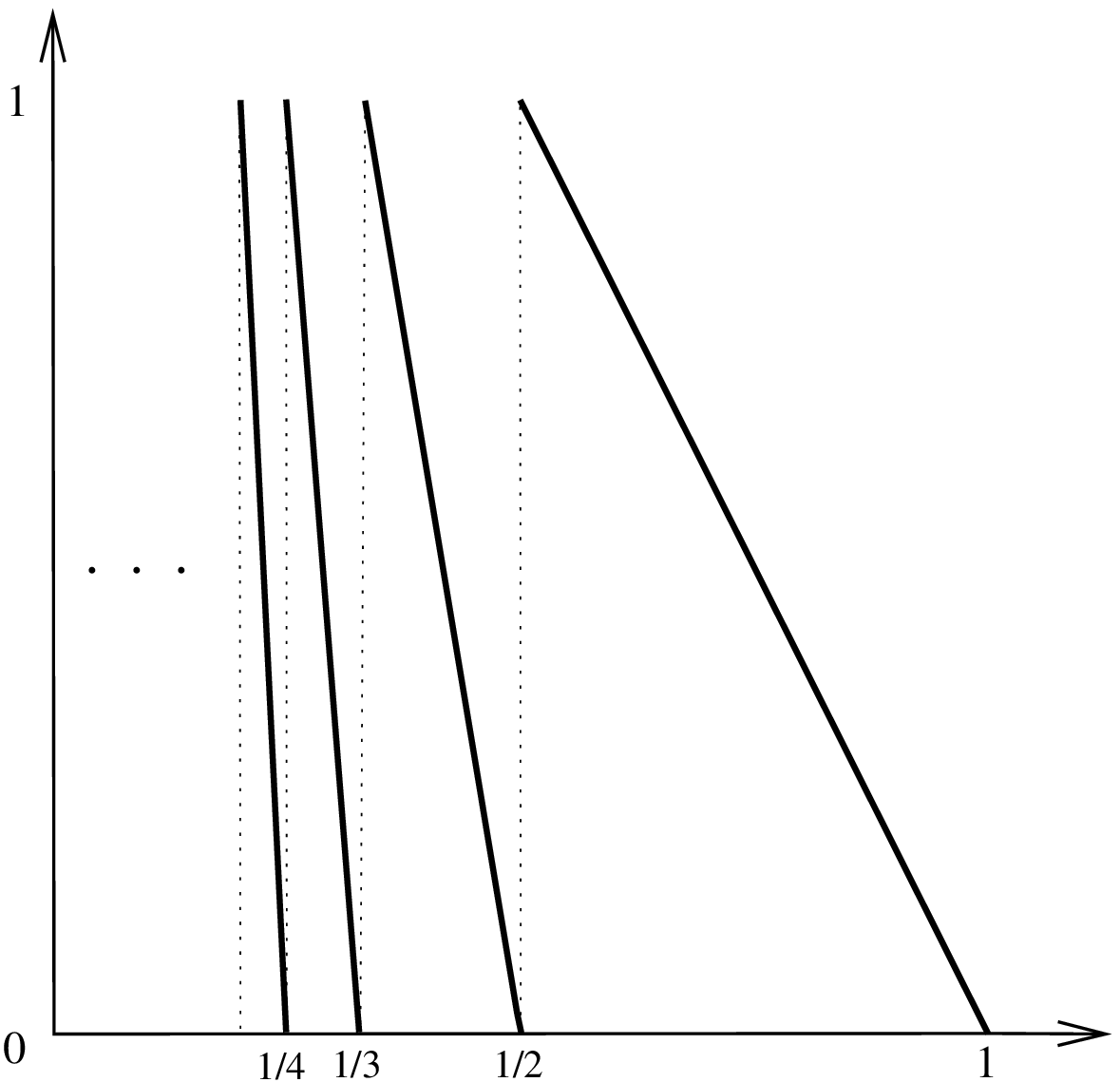}\hspace{0.1\textwidth}
\includegraphics[width=0.42\textwidth]{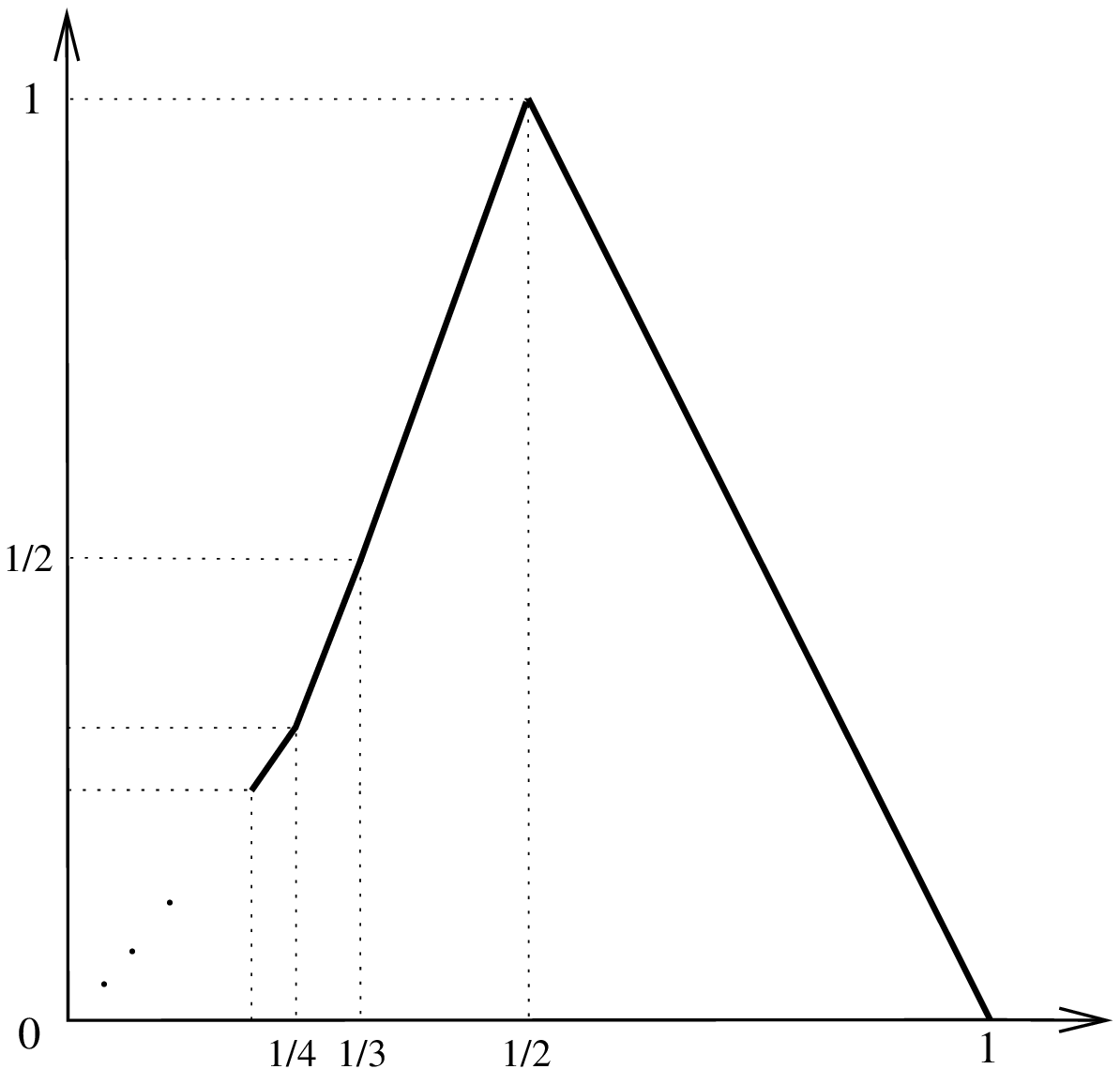}
\caption{The $\alpha_H$-L\"uroth and $\alpha_H$-Farey map, where $t_{n}=1/n$, $n\in \N$.}\end{center}
\end{figure}\label{fig:ClassicalLF}


Let us now describe how to construct a Markov partition ${\mathcal{A}}$ from the partition $\alpha$, and its associated coding for the map
$F_\alpha$. (For the definition of a Markov partition, see, for instance, \cite{tseng}.) The partition $\mathcal{A}$
is given by the closed intervals $\{A, B\}$, where $A:=\overline{A_1}$ and $B:={\U\setminus A_1}$. Each
$\alpha$-irrational number
in $\U$ has an infinite coding
$x=\langle x_{1}, x_{2},\ldots \rangle_{\alpha} \in \{0,1\}^{\N}$,
which is given by $x_k=1$ if and only if $F_{\alpha}^{k-1}(x)\in \mathrm{Int}(A)$
for each $k \in \N$.
This coding will be referred to
 as the $\alpha$-Farey coding.
If an $\alpha$-irrational number $x\in\U$ has $\alpha$-L\"uroth coding given by $x=[\ell_1, \ell_2, \ell_3, \ldots]_\alpha$, then the $\alpha$-Farey coding of $x$ is given by
$x=\langle0^{\ell_1-1},1,0^{\ell_2-1},1,0^{\ell_3-1},1,\ldots\rangle_\alpha$, where $0^{n}$ denotes the sequence of $n$ consecutive appearances of the symbol $0$, whereas for each $\alpha$-rational number $x=[\ell_1, \ell_2,\ldots, \ell_k]_\alpha$, one immediately verifies that this number has an $\alpha$-Farey coding given either by \[x=\langle 0^{\ell_1-1},1,0^{\ell_2-1},1,\ldots,0^{\ell_k-1},1,0,0,0,\ldots\rangle_\alpha\] or \[
x=\langle 0^{\ell_1-1},1,0^{\ell_2-1},1,\ldots,0^{\ell_k-2},1,1,0,0,0,\ldots\rangle_\alpha.\]
Let us now define the cylinder sets associated with the map $F_\alpha$. These coincide with the refinements $\mathcal{A}^n$ of the partition $\mathcal{A}$ for $F_\alpha$.
For each $n$-tuple $(x_1, \ldots, x_n)$ of positive integers, define the \textit{$\alpha$-Farey cylinder set }  $\widehat{C}_{\alpha}(x_{1},\ldots,x_{n})$ by setting
    \[
\widehat{C}_{\alpha}(x_{1},\ldots,x_{n}):=\{\langle y_{1},y_{2}, \ldots \rangle_{\alpha}: y_{k}=x_{k}, \text{ for }1\leq k\leq n\}.
\]
Notice that every $\alpha$-L\"uroth cylinder set is also an
$\alpha$-Farey cylinder set, whereas the converse of this statement is not true. The precise description of the correspondence is that   any
 $\alpha$-Farey cylinder set which has the form
$\widehat{C}_\alpha(0^{\ell_{1}-1},1,
\ldots ,0^{\ell_{k}-1},1)$ coincides with the
$\alpha$-L\"uroth cylinder set
$C_\alpha(\ell_1, \ldots, \ell_k)$ but if an
$\alpha$-Farey cylinder set is defined by a finite word ending in the symbol $0$, then it
 cannot be  translated to a single $\alpha$-L\"uroth cylinder set.  However, we do have the relation
\[
\widehat{C}_\alpha(0^{\ell_{1}-1},1,0^{\ell_{2}-1},1,
\ldots ,0^{\ell_{k}-1},1,0^{m})=\bigcup_{n\geq m+1}C_\alpha(\ell_1, \ell_2, \ldots, \ell_k, n).
\]
It therefore follows that for the Lebesgue measure of this interval we have that
\begin{eqnarray*}
\lambda(\widehat{C}_\alpha(0^{\ell_{1}-1},1,0^{\ell_{2}-1},1,
\ldots ,0^{\ell_{k}-1},1,0^{m}))&=& \sum_{n\geq m+1}\lambda(C_\alpha(\ell_1, \ell_2, \ldots, \ell_k, n))\\&=&a_{\ell_1}a_{\ell_2}\cdots a_{\ell_k}t_{m+1}.
\end{eqnarray*}
In addition, we can identify the endpoints of each $\alpha$-Farey cylinder set. If we consider the cylinder set  $\widehat{C}_\alpha(0^{\ell_{1}-1},1,\ldots, 0^{\ell_{k}-1},1)$, then we already know the endpoints of this interval (since it is also equal to an $\alpha$-L\"uroth cylinder set). On the other hand, the endpoints of the set $\widehat{C}_\alpha(0^{\ell_{1}-1},1,0^{\ell_{2}-1},1,\ldots,0^{\ell_{k}-1},1, 0^{m})$ are given by $[\ell_1, \ldots, \ell_{k}, m+1]_\alpha$ and $[\ell_1, \ldots, \ell_{k}]_\alpha$. 

The following result concerning the $\alpha$-Farey system and the tent system was obtained in \cite{KMS}. Before stating it, we remind the reader that
the measure of maximal entropy $\mu_{\alpha}$ for the system $F_\alpha$
is the measure that assigns mass $2^{-n}$ to each $n$-th level
$\alpha$-Farey cylinder set, for each $n\in\N$. Also, we recall that the distribution function $\Delta_\mu$ of a measure $\mu$ with support in $[0,1]$ is defined for each $x\in[0,1]$ by
   \[
\Delta_\mu(x):=\mu([0,x)).
\]

\vspace{1mm}

\begin{lem}[\cite{KMS}, Lemma 2.2]\label{KMSlem}
The dynamical systems $(\U, {F}_{\alpha})$ and
$(\U, T)$ are topologically conjugate and
the conjugating
homeomorphism is given, for each $x=[ \ell_1, \ell_2, \ldots]_{\alpha}$, by
\[
{\theta_{\alpha}}(x):=-2\sum_{k=1}^\infty(-1)^k2^{-\sum_{i=1}^k \ell_i}.
\]
Moreover, the map $\theta_{\alpha}$ is equal to the distribution function of
the measure of maximal entropy  $\mu_{\alpha}$  for the $\alpha$-Farey map.
\end{lem}

Let us remark that  this map should be seen as an analogue of Minkowski's question-mark function, which was originally introduced by Minkowski \cite{min} in order to illustrate the Lagrange property of algebraic numbers of degree two. Indeed, all that is different in the definition of each is that in Minkowski's function the continued fraction entries appear and in the function $\theta_\alpha$, these are replaced by the $\alpha$-L\"uroth entries. For this reason,  we refer to the map $\theta_\alpha$ as the {\em $\alpha$-Farey-Minkowski function}.

\section{Differentiability properties of $\theta_\alpha$}

In this section we will give a series of simple lemmas that describe the differentiability properties of the function $\theta_\alpha$ for an arbitrary partition $\alpha$. The results turn out to match the results for the Minkowski question-mark function, although a little care must be taken when dealing with certain partitions. Most of the proofs here are modelled after the corresponding proofs in \cite{mink?}.  Before we begin, though, let us point out that (as mentioned above) the tent map itself is an example of an $\alpha$-Farey map, coming from the dyadic partition $\alpha_D:=\{(1/2^n, 1/2^{n-1}]:n\in\N\}$. Obviously, then, the map $\theta_{\alpha_D}$ which conjugates the map $F_{\alpha_D}$ and the tent map is simply the identity. So, in this case, the derivative of $\theta_{\alpha_D}$ is clearly identically equal to 1. So, in what follows,  unless otherwise stated, $\alpha$ is understood to be an arbitrary partition of the form detailed in the introduction but we also assume that $\alpha$ is {\em non-dyadic}, that is, we assume that $\alpha$ is not equal to the partition $\alpha_D$.

\vspace{1mm}

In order to state the first lemma, we must first make the following definition. For an $\alpha$-irrational number $x\in\U$ and for each $n\in\N$,
define the interval $I_n^{(\alpha)}(x)$ to be the unique $n$-th level $\alpha$-Farey cylinder set that contains the point $x$.  Let us also remind the reader here that we use the phrase ``exists in a generalised sense'' to mean ``exists or is equal to infinity''.

\begin{lem}\label{lemma1}
Suppose that $x\in[0,1]$ is such that $\theta_\alpha'(x)$ exists in a generalised sense. We then have that:
\begin{itemize}
  \item [(a)] If $x=[\ell_1, \ell_2, \ldots]_\alpha$ is an $\alpha$-irrational number, then
  \[
  \theta_\alpha'(x)=\lim_{n\to\infty}\frac{2^{-n}}{\lambda\left(I_n^{(\alpha)}(x)\right)}.
  \]
  \item [(b)] If $x=[\ell_1, \ldots, \ell_k]_\alpha$ is an $\alpha$-rational number, then
  \[
  \theta_\alpha'(x)=\frac{2\cdot 2^{-(\ell_1+\cdots +\ell_k)}}{a_{\ell_1}\ldots a_{\ell_k}}\lim_{m\to\infty}\frac{2^{-m}}{t_m}.
  \]
\end{itemize}
\end{lem}

\begin{proof}
To prove part (a), let $x$ be an $\alpha$-irrational number such that $\theta_\alpha'(x)$ exists in a generalised sense. Then for every sequence $(y_n)_{n\geq1}$ decreasing or increasing to $x$, we have that
\[
\lim_{n\to\infty}\frac{\theta_\alpha(y_n)-\theta_\alpha(x)}{y_n-x}=\theta_\alpha'(x).
\]
In particular this holds if we consider the sequences of endpoints of the intervals $I_n^{(\alpha)}(x):=[L_n, R_n]$ which approach $x$ from the left and right, respectively. So, letting $A_n:=\theta_\alpha(x)-\theta_\alpha(L_n)$, $B_n:=x-L_n$, $C_n:=\theta_\alpha(R_n)-\theta_\alpha(x)$ and $D_n:=R_n-x$, we have that
\[
\lim_{n\to\infty}\frac{A_n}{B_n}=\lim_{n\to\infty}\frac{C_n}{D_n}=\theta_\alpha'(x).
\]
It then follows easily that
\[
\lim_{n\to\infty}\frac{2^{-n}}{\lambda(I_n^{\alpha}(x))}=\lim_{n\to\infty}\frac{\theta_\alpha(R_n)-\theta_\alpha(L_n)}{R_n-L_n}=\lim_{n\to\infty}\frac{A_n+C_n}{B_n+D_n}=\theta_\alpha'(x).
\]
This finishes the proof of part (a).

For part (b), let $x=[\ell_1, \ldots, \ell_k]_\alpha$ and again suppose that $\theta_\alpha'(x)$ exists in a generalised sense. Then, just as in the $\alpha$-irrational case above, for the sequence $([\ell_1, \ldots, \ell_k, m]_\alpha)_{m\geq1}$ which approaches the point $x$, we have that
\[
\lim_{m\to\infty}\frac{\theta_\alpha([\ell_1, \ldots, \ell_k, m]_\alpha)-\theta_\alpha([\ell_1, \ldots, \ell_k]_\alpha)}{[\ell_1, \ldots, \ell_k, m]_\alpha-[\ell_1, \ldots, \ell_k]_\alpha} =\lim_{m\to\infty}\frac{2\cdot 2^{-(\ell_1+\cdots\ell_k+m)}}{a_{\ell_1}\ldots a_{\ell_k}t_{m}}=\frac{2\cdot 2^{-(\ell_1+\cdots\ell_k)}}{a_{\ell_1}\ldots a_{\ell_k}}\lim_{m\to\infty}\frac{2^{-m}}{t_m}.
\]
\end{proof}

We now come to the question of the particular values the derivative of $\theta_\alpha$ may take, if it exists. The answer is given in Proposition \ref{0inf} below, but before we can get there we need the following two lemmas.

\begin{lem}\label{0infnot1/2}
Suppose that the partition $\alpha$ is such that $a_1\neq 1/2$. Let $x$ be an $\alpha$-irrational number with the property that $\theta_\alpha'(x)$ exists in a generalised sense. Then,
\[
\theta_\alpha'(x)\in\{0, \infty\}.
\]

\end{lem}

\begin{proof}Let $x$ be an $\alpha$-irrational number and suppose that $\theta_\alpha'(x)$ exists  in a generalised sense.
By Lemma \ref{lemma1}, it follows that
\[
  \theta_\alpha'(x)=\lim_{n\to\infty}\frac{2^{-n}}{\lambda\left(I_n^{(\alpha)}(x)\right)}.
  \]
Suppose, by way of contradiction, that $\theta_\alpha'(x)=c$, for $0<c<\infty$. Then, it follows that
\[
\lim_{n\to\infty}\frac{2^{-n}}{\lambda\left(I_n^{(\alpha)}(x)\right)}\cdot\frac{\lambda\left(I_{n-1}^{(\alpha)}(x)\right)}{2^{-(n-1)}}=1
\]
and consequently,
\begin{eqnarray}\label{star}
\lim_{n\to\infty}\frac{\lambda\left(I_{n-1}^{(\alpha)}(x)\right)}{\lambda\left(I_n^{(\alpha)}(x)\right)}=2.
\end{eqnarray}
Since $x$ is an $\alpha$-irrational number, it follows that, infinitely often, the $n$-th $\alpha$-Farey cylinder set containing the point $x$ is also an $\alpha$-L\"uroth cylinder set. More specifically, where $x=[\ell_1, \ell_2, \ell_3, \ldots]_\alpha$,  these are the following sets:
\[
I_{\ell_1}^{(\alpha)}(x),I_{\ell_1+\ell_2}^{(\alpha)}(x), I_{\ell_1+\ell_2+\ell_3}^{(\alpha)}(x), \ldots.
\]
Recall that we have $\lambda\left(I_{\sum_{i=1}^n\ell_i}^{(\alpha)}(x)\right)=a_{\ell_1}\ldots a_{\ell_n}$. Furthermore,
\[
\lambda\left(I_{(\sum_{i=1}^n\ell_i)+1}^{(\alpha)}(x)\right):=\left\{
                                                                \begin{array}{ll}
                                                                  a_{\ell_1}\ldots a_{\ell_n}a_1, & \hbox{if $\ell_{n+1}=1$;} \\
                                                                  a_{\ell_1}\ldots a_{\ell_n}t_2, & \hbox{if $\ell_{n+1}>1$.}
                                                                \end{array}
                                                              \right.
\]
 Thus, for each $n\in\N$, the quotient $\lambda\left(I_{\sum_{i=1}^n\ell_i}^{(\alpha)}(x)\right)/\lambda\left(I_{(\sum_{i=1}^n\ell_i)+1}^{(\alpha)}(x)\right)$ is either equal to $1/a_1$ or $1/t_2$. Given that $a_1\neq1/2$, neither $1/a_1$ nor $1/t_2$ can be equal to 2. This contradicts (\ref{star}) and the proof is finished.
\end{proof}

The above proof is closely modelled on the corresponding result for the Minkowski question-mark function given in \cite{mink?}. The problem with the situation where $a_1=1/2$ can be overcome with the help of the next lemma.

\begin{lem}\label{0inf1/2}
Suppose that there exists some proper subset $M\subset\N$ such that for all $i\in M$, the partition $\alpha$ satisfies $a_i=2^{-i}$ and for all $i\in \N\setminus M$, we have that $a_i\neq 2^{-i}$. Define the sets
\[
B_{M, N}:=\{x\in[0,1]:\ell_i(x)\in M \text{ for all } i\geq N\} \ \text{ and }\  B_M:=\bigcup_{N\in\N}B_{M, N}.
\]
Then, if $x\in B_M$, we have that $\theta_\alpha'(x)$ does not exist.
\end{lem}

\begin{proof}

We will prove the lemma by considering two separate  cases. The first case is that $a_1\neq 1/2$, or, in other words, the set $M$ does not contain the number 1. Fix $N\in\N$ and suppose, by way of contradiction, that $x\in B_{M, N}$ and that the derivative $\theta_\alpha'(x)$ does exist  in a generalised sense. Then, by Lemma \ref{lemma1}, we know that
\[
\theta_\alpha'(x)=\lim_{n\to\infty}\frac{2^{-n}}{\lambda\left(I_n^{(\alpha)}(x)\right)}.
\]
Also, since $a_1\neq 1/2$, we know that $\ell_i\neq 1$ for all $i\geq N$. Therefore, in the $\alpha$-Farey coding for $x$, after $\sum_{i=1}^{N-1}\ell_i$ entries, every occurrence of a 1 is followed directly by a 0. Let us choose two subsequences from the sequence $\left(2^{-n}/\lambda\left(I_n^{(\alpha)}(x)\right)\right)_{n\geq1}$. For the first, pick out every $n\geq \sum_{i=1}^{N-1}\ell_i$ such that
the $\alpha$-Farey interval $I_n^{(\alpha)}(x)$ is also an $\alpha$-L\"uroth interval (that is, pick out the elements of the sequence that correspond to stopping at every 1 in the $\alpha$-Farey code of $x$). For the second, take the subsequence that corresponds to shifting the first subsequence by exactly one place forward. Therefore we have the following two sequences, which, according to Lemma \ref{lemma1} ought to have the same limit:
\[
\frac{2^{-\sum_{i=1}^{N-1}\ell_i}}{a_{\ell_1}\ldots a_{\ell_N}}\left(\frac{2^{-\ell_N}}{a_{\ell_N}}, \frac{2^{-(\ell_N+\ell_{N+1})}}{a_{\ell_N}a_{\ell_{N+1}}},\frac{2^{-(\ell_N+\ell_{N+1}+\ell_{N+2})}}{a_{\ell_N}a_{\ell_{N+1}}a_{\ell_{N+2}}},  \ldots \right)
\]
and
\[
\frac{2^{-\sum_{i=1}^{N-1}\ell_i}}{a_{\ell_1}\ldots a_{\ell_N}}\left(\frac{2^{-(\ell_N+1)}}{a_{\ell_N}t_2}, \frac{2^{-(\ell_N+\ell_{N+1}+1)}}{a_{\ell_N}a_{\ell_{N+1}}t_2}, \frac{2^{-(\ell_N+\ell_{N+1}+\ell_{N+2}+1)}}{a_{\ell_N}a_{\ell_{N+1}}a_{\ell_{N+2}}t_2},\ldots \right).
\]
However,  notice that since $a_{\ell_{N+m}}=2^{-\ell_{N+m}}$ for all $m\geq0$,  we have that
\begin{eqnarray}\label{star3}
\lim_{m\to\infty}\frac{2^{-(\ell_N+\cdots+\ell_{N+m})}}{a_{\ell_N}\ldots a_{\ell_{N+m}}}=1,
\end{eqnarray}
whereas,
\[
\lim_{m\to\infty}\frac{2^{-(\ell_N+\cdots+\ell_{N+m}+1)}}{a_{\ell_N}\ldots a_{\ell_{N+m}}t_2}=  \frac{2^{-1}}{t_2}\neq 1.
\]
Consequently the derivative of $\theta_\alpha$ at $x$ does not exist.

To finish the proof, we consider the case where $a_1=1/2$.
First notice that the argument in (\ref{star3}) obviously still holds whenever a point $x$ is such that eventually all the $\alpha$-L\"uroth entries lie in the set $M$. Without loss of generality, we suppose that every $\ell_i(x)\in M$. This implies, in light of Lemma \ref{lemma1}, that if the derivative $\theta_\alpha'(x)$ exists in a generalised sense, then it must be equal to 1. We must again consider two further cases. The first is the case that for every $k\in M$, not only does $a_k=2^{-k}$, but also $t_{k+1}=2^{-k}$. The second case is that this is no longer true, in other words, there exists $k\in M$ such that $t_{k+1}\neq 2^{-k}$.

Consider first the situation that $t_{k+1}=2^{-k}$ for all $k\in M$. It then follows from a simple calculation that if $x$ is such that every $\alpha$-L\"uroth digit $\ell_i(x)$ of $x$ belongs to the set  $M$, then $\theta_\alpha(x)=x$. Now suppose that $M$ is a bounded set, with largest element $k$. It therefore follows that $a_{k+1}\neq 2^{-(k+1)}$ and $t_{k+2}\neq 2^{-(k+1)}$.  Since all entries of $x$ lie in $M$ and, in particular, $\ell_{2n}\leq k$, the sequence $([\ell_1, \ldots, \ell_{2n-1}, k+2]_\alpha)_{n\geq1}$ tends to $x$ from above. Therefore, we have (provided that the limit exists),
\begin{eqnarray*}
\lim_{n\to \infty}\frac{\theta_\alpha([\ell_1, \ldots, \ell_{2n-1}, k+2]_\alpha)-\theta_\alpha(x)}{[\ell_1, \ldots, \ell_{2n-1}, k+2]_\alpha-x}&=&\lim_{n\to\infty}\frac{[\ell_{2n}, \ell_{2n+1}, \ldots]_\alpha-2^{-(k+1)}}{[\ell_{2n}, \ell_{2n+1}, \ldots]_\alpha-t_{k+2}}\\&=&
1+\left(t_{k+2}-2^{-(k+1)}\right)\lim_{n\to\infty}\frac{1}{[\ell_{2n}, \ell_{2n+1}, \ldots]_\alpha-t_{k+2}}\neq1.
\end{eqnarray*}
Thus, in this case we also have that the derivative of $\theta_\alpha$ at $x$ does not exist.
Suppose now that $M$ is unbounded. Then, there must exist a smallest integer $k\geq1$ such that $a_{k+1}\neq 2^{-(k+1)}$. From this, it follows that also $t_{k+2}\neq 2^{-(k+1)}$. If $x=[\ell_1, \ell_2, \ldots]_\alpha$ is such that eventually all the digits of $x$ lie in $M$ but are at most equal to $k$, we can show that the derivative of $\theta_\alpha$ at $x$ does not exist exactly as above, where $M$ was assumed to be bounded. So, suppose that there exists a subsequence $(\ell_{i_j})_{j\geq1}$ of the entries of $x$ with each $\ell_{i_j}\in M$ and $\ell_{i_j}>k+1$. Further, suppose that each of these entries appear in even positions (if odd, the proof can be easily modified accordingly). Consider the sequence
 $(A_{i_n}:=[\ell_1, \ldots, \ell_{i_n-1}, k+2]_\alpha)_{n\geq1}$, which tends to $x$ from below. We then obtain that
\[
\frac{\theta_\alpha(x)-\theta_\alpha(A_{i_n})}{x-A_{i_n}}=\frac{2\cdot 2^{-(k+2)}-[\ell_{i_n}, \ell_{i_n+1}, \ldots]}{t_{k+2}-[\ell_{i_n}, \ell_{i_n+1}, \ldots]}
\]
and so (if the limit exists at all), $\lim_{n\to \infty}\frac{\theta_\alpha(x)-\theta_\alpha(A_{i_n})}{x-A_{i_n}}\neq 1$.
Therefore, in this second subcase we have also shown that the derivative of $\theta_\alpha$ at $x$ does not exist. This finishes the proof of the first subcase.

Finally, we must consider the situation where there exists at least one $k\in M$ such that $t_{k+1}\neq 2^{-k}$. Observe that if $k\in M$ is such that $t_{k+1}\neq 2^{-k}$, it follows that $t_k\neq 2^{-(k-1)}$ also. Recall that since we are assuming that $a_1=1/2$, that $t_2=1/2$ and so this $k$ cannot be equal to 1. This case only needs special consideration whenever the $\alpha$-L\"uroth code of $x$ contains a sequence of entries $\ell_{i_j}(x)$ with $t_{\ell_{i_j}+1}\neq 2^{-\ell_{i_j}}$ for all $j\in \N$. Suppose that this holds and, where we have set $n_j:=\left(\sum_{k=1}^{i_j}\ell_{k}(x)\right)-1$,  consider the following limit (provided that it exists):
\[
\lim_{j\to\infty} \frac{2^{-n_j}}{\lambda\left(I_{n_j}^{(\alpha)}(x)\right)}=\lim_{j\to \infty} \frac{2^{-(\ell_{i_j}-1)}}{t_{\ell_{i_j}}}\neq 1.
\]
Therefore, in this final subcase, we have demonstrated that the derivative $\theta_\alpha'(x)$ does not exist. This finishes the proof.
\end{proof}

We now give the main result of this section.

\begin{prop}\label{0inf}
For an arbitrary non-dyadic partition $\alpha$, if $x\in[0,1]$ is  such that the derivative $\theta_\alpha'(x)$ exists in a generalised sense, then
\[
\theta_\alpha'(x)\in\{0, \infty\}.
\]
\end{prop}

\begin{proof}
In light of part (b) of Lemma \ref{lemma1}, it suffices to consider $\alpha$-irrational numbers. If the partition $\alpha$ is such that $a_1\neq1/2$, then the result follows by Lemma \ref{0infnot1/2}. It therefore only remains to consider the situation where $a_1=1/2$.  In this case  we can say that there exists a proper subset $M\subset \N$, which at least contains the point $1$, such that $a_i=2^{-i}$ for all $i\in M$. Since we do not allow $\alpha$ to be the dyadic partition, there exist at least two indices $i\in \N$ such that $a_i\neq 2^{-i}$.  By Lemma \ref{0inf1/2}, we know that if $x$ belongs to the set $B_M$ (which is defined in the statement of said lemma), then the derivative of $\theta_\alpha$ does not exist at the point $x$. This further reduces the remaining work to the situation where $x\notin B_M$.

To complete the proof, then, we consider a further two subcases. For the first, suppose that $k$ is the smallest integer such that $a_k\neq 2^{-k}$ and let $x\in[0,1]$ be such that $\theta_\alpha'(x)$ exists in a generalised sense and $\ell_i(x)\in M \cup \{k\}$, with infinitely many of the $\ell_i$ equal to $k$. Then, once again we use Lemma \ref{lemma1} to obtain
\[
\theta_\alpha'(x)=\lim_{n\to\infty}\frac{2^{-n}}{\lambda\left(I_n^{(\alpha)}(x)\right)}.
\]
In particular,
\[
\theta_\alpha'(x)=\lim_{n\to\infty}\frac{2^{-\sum_{i=1}^n\ell_i}}{a_{\ell_1}\ldots a_{\ell_{n}}}.
\]
However, for each $\ell_i(x)$ such that $\ell_i\in M$, we have that $2^{-\ell_i}/a_{\ell_i}=1$, so the above limit reduces to
\[
\lim_{n\to\infty} \left(\frac{2^{-k}}{a_k}\right)^n = \left\{
                                                        \begin{array}{ll}
                                                          0, & \hbox{if $2^{-k}<a_k$;} \\
                                                          \infty, & \hbox{$2^{-k}>a_k$.}
                                                        \end{array}
                                                      \right.
\]
Thus, in this first subcase, given that we assume that $a_k\neq 2^{-k}$,  if the derivative of $\theta_\alpha$ exists in a generalised sense at $x$, it must lie in the set $\{0, \infty\}$.

For the second subcase, again suppose that $k$ is the smallest integer such that $a_k\neq 2^{-k}$ and observe that this also means that $t_{k+1}\neq 2^{-k}$. Now suppose that $x$ is such that infinitely often $\ell_i(x)>k$ and that also infinitely many of the $\alpha$-L\"uroth entries of $x$ are outside the set $M$ (in case there are further indices $k+n$ with $a_{k+n}=2^{-(k+n)}$, since if the entries of $x$ would stay all but finitely often in $M$, we are back to the situation of Lemma \ref{0inf1/2}). Then, just as above, we have that
\[
\theta_\alpha'(x)=\lim_{n\to\infty}\frac{2^{-\sum_{i=1}^n\ell_i}}{a_{\ell_1}\ldots a_{\ell_{n}}}.
\]
Since each time $\ell_i\in M$, we are only multiplying by 1 in this sequence, without loss of generality, we may suppose that $\ell_i\notin M$ for all $i\in \N$. Given this assumption, we can write the $\alpha$-L\"uroth code for $x$ in the following way:
\[
x=[\underbrace{k, \ldots, k}_{n_1 \mbox{\scriptsize{ times}}}, \ell_{i_1},\underbrace{k, \ldots, k}_{n_2 \mbox{\scriptsize{ times}}}, \ell_{i_2}, \ldots, \underbrace{k, \ldots, k}_{n_j \mbox{\scriptsize{ times}}}, \ell_{i_j}, k, \ldots]_\alpha,
\]
where $n_j\in \N\cup\{0\}$ and  $\ell_{i_j}\geq k+1$, for all $j\in\N$. We then obtain the limit
\[
\theta_\alpha'(x)=\lim_{j\to\infty}\frac{2^{-((n_1+\cdots n_j)k+(\ell_{i_1}+\cdots+\ell_{i_{j-1}})+k)}}{(a_k)^{n_1+\cdots n_j}a_{\ell_{i_1}}\ldots a_{\ell_{i_{j-1}}}t_{k+1}}=\frac{2^{-k}}{t_{k+1}}\lim_{m\to \infty}\frac{2^{-\sum_{i=1}^m\ell_i}}{a_{\ell_1}\ldots a_{\ell_{m}}},
\]
where $m=(n_1+1)+(n_2+1)+\cdots + (n_{j-1}+1)+n_j$. Thus, we have that
\[
\theta_\alpha'(x)=\frac{2^{-k}}{t_{k+1}}\theta_\alpha'(x).
\]
It therefore follows that $\theta_\alpha'(x)$, whenever it exists in a generalised sense, belongs to the set $\{0, \infty\}$, since $2^{-k}/t_{k+1}\neq1$. This finishes the proof.
\end{proof}

We now give the result that for each non-dyadic partition $\alpha$, the $\alpha$-Farey-Minkowski function is singular with respect to the Lebesgue measure. Recall that this means that the derivative of $\theta_\alpha$ is Lebesgue-a.e. equal to zero. This mirrors the well-known result of Salem \cite{Salem} that Minkowski's question-mark function is singular.

\begin{prop}\label{singular}
For an arbitrary non-dyadic partition $\alpha$, the function $\theta_\alpha$ is singular with respect to the Lebesgue measure.
\end{prop}

\begin{proof}
In light of Lemma \ref{KMSlem}, we know that each function $\theta_\alpha$ is strictly increasing. Therefore, by a classical theorem (see, for instance, Theorem 5.3 in \cite{royden}), the derivative of $\theta_\alpha$ exists and is in particular finite $\lambda$-almost everywhere. Therefore, by Lemmas \ref{0inf1/2} and \ref{0infnot1/2}, it follows that the derivative is equal to 0 for $\lambda$-a.e. $x\in\U$.
\end{proof}

From this point on, we have to assume a little more information about the partitions $\alpha$. We will henceforth assume that all partitions $\alpha$ are  either {\em expansive of exponent $\tau>0$}, or, {\em expanding}.
Recall the definitions from \cite{KMS}: A partition $\alpha$ is said to be expansive of exponent $\tau>0$ if for the tails
of $\alpha$ we have that $t_n=n^{-\tau}\psi(n)$, for some slowly-varying\footnote{A measurable function $f:\R^{+} \to \R^{+}$ is said
to be
{\em slowly varying}  if $
\lim_{x\to\infty}f(x y)/f(x)=1$, for all $y>0$.} function $\psi:\N\to\R^+$, whereas $\alpha$ is said to be expanding if $\lim_{n\to \infty} t_{n}/t_{n+1}= \rho$, for some $\rho>1$.

Before stating the next proposition, let us define the $k$-th approximant $r_k^{(\alpha)}(x)$ to a point $x=[\ell_1,\ell_2,\ldots]_\alpha$ by setting $r_k^{(\alpha)}(x):=[\ell_1, \ldots, \ell_k]_\alpha$. We will use the notation $[a, b]_\pm$ to indicate that we either have the interval $[a, b]$ or the interval $[b, a]$, depending on which number is larger.  Finally, recall that the measure $\mu_\alpha$ by definition  gives mass $2^{-n}$ to each $n$-th level $\alpha$-Farey cylinder set.

\begin{prop}\label{liminfderinf}
Let $\alpha$ be a partition that is either { expansive of exponent $\tau>0$}  or { expanding}. Suppose that $x$ is such that
\[
\lim_{k\to\infty}\frac{\mu_\alpha\left(\left[r_k^{(\alpha)}(x), r_{k+1}^{(\alpha)}(x)\right]_{\pm}\right)}{\lambda\left(\left[r_k^{(\alpha)}(x), r_{k+1}^{(\alpha)}(x)\right]_{\pm}\right)}=\infty.
\]
Then $\theta_\alpha'(x)=\infty$.
\end{prop}

\begin{proof}
First, notice that we have $[r_k^{(\alpha)}(x), r_{k+1}^{(\alpha)}(x)]_\pm=I_{\ell_1+\cdots \ell_{k+1}-1}^{(\alpha)}$, from which we immediately deduce that
\[
\frac{\mu_\alpha\left(\left[r_k^{(\alpha)}(x), r_{k+1}^{(\alpha)}(x)\right]_\pm\right)}{\lambda\left(\left[r_k^{(\alpha)}(x), r_{k+1}^{(\alpha)}(x)\right]_\pm\right)}=\frac{2\cdot2^{-(\ell_1+\cdots +\ell_{k+1})}}{a_{\ell_1}\ldots a_{\ell_{k}}t_{\ell_{k+1}}}.
\]
Let $y>x$. Then, for all $y$ close enough to $x$, there exists an even positive integer $k$ such that
\begin{eqnarray}\label{star1}
y&\in& ([\ell_1, \ldots, \ell_{k+1}]_\alpha, [\ell_1, \ldots, \ell_{k-1}]_\alpha].
\end{eqnarray}
We will consider separately the cases $\ell_{k+1}>1$ and $\ell_{k+1}=1$. Suppose that we are the first of these cases, so $\ell_{k+1}>1$. Then we can split the interval in (\ref{star1}) up into smaller intervals to locate $y$ with greater precision. Between the points $[\ell_1, \ldots, \ell_{k+1}]_\alpha$ and $[\ell_1, \ldots, \ell_{k-1}]_\alpha$ lie the points (written in increasing order), $[\ell_1, \ldots, \ell_{k}, \ell_{k+1}-1]_\alpha$, $[\ell_1, \ldots, \ell_{k}, \ell_{k+1}-2]_\alpha$, \ldots, $[\ell_1, \ldots, \ell_{k}, 2]_\alpha$ and $[\ell_1, \ldots, \ell_{k}, 1]_\alpha$, as shown in Figure 3.1 below.

\begin{figure}[htbp]
\begin{center}
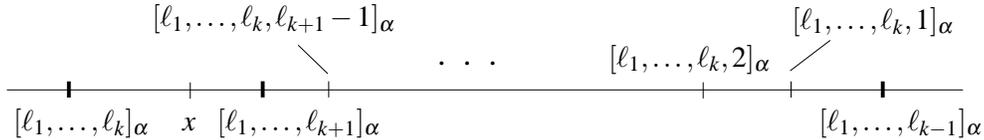\caption{The positions of the convergents (indicated with thicker lines) and intermediary points to $x$. }
\end{center}
\end{figure}

\noindent{\bf Case 1.1} $\ $Suppose first that
\[
[\ell_1, \ldots, \ell_{k}, 1]_\alpha<y\leq [\ell_1, \ldots, \ell_{k-1}]_\alpha.
\]
Immediately from the definition of $\theta_\alpha$ and the fact that $\theta_\alpha$ is an increasing function, we calculate that
\begin{eqnarray*}
\theta_\alpha(y)-\theta_\alpha(x)&\geq&\theta_\alpha([\ell_1, \ldots, \ell_{k}, 1]_\alpha)-\theta_\alpha([\ell_1, \ldots, \ell_{k+1}]_\alpha)\\
&=& 2\cdot 2^{-(\ell_1+\cdots +\ell_k)}(2^{-1}-2^{-\ell_{k+1}})\\
&\gg &2\cdot2^{-(\ell_1+\cdots +\ell_k)},
\end{eqnarray*}
where the last inequality comes from the fact that $\ell_{k+1}>1$. Moreover,
\[
y-x \leq [\ell_1, \ldots, \ell_{k-1}]_\alpha-[\ell_1, \ldots, \ell_{k}]_\alpha=a_{\ell_1}\ldots a_{\ell_{k-1}}t_{\ell_{k}}.
\]
Thus, in this first instance, we obtain that
\[
\frac{\theta_\alpha(y)-\theta_\alpha(x)}{y-x}\gg \frac{2\cdot2^{-(\ell_1+\cdots +\ell_k)}}{a_{\ell_1}\ldots a_{\ell_{k-1}}t_{\ell_{k}}}.
\]

\vspace{2mm}

\noindent{\bf Case 1.2} $\ $Now suppose that there exists a positive integer $n\in\{1, 2, \ldots, \ell_{k+1}-2\}$ such that
\[
[\ell_1, \ldots, \ell_{k}, n+1]_\alpha<y\leq [\ell_1, \ldots, \ell_{k}, n]_\alpha.
\]


At this point we have to split the argument up again. First suppose that the partition $\alpha$ is either expansive with exponent $\tau$, or, expanding with $\lim_{n\to \infty} t_{n}/t_{n+1}= \rho$ for $1<\rho<2$. We then  obtain that
\begin{eqnarray*}
\frac{\theta_\alpha(y)-\theta_\alpha(x)}{y-x}&\gg&\frac{2^{-(\ell_1+\cdots +\ell_{k+1})}}{a_{\ell_1}\ldots a_{\ell_{k}}t_{\ell_{k+1}}}\cdot\frac{2^{\ell_{k+1}-(n+1)}t_{\ell_{k+1}}}{t_n}\\
&\gg& \frac{2^{-(\ell_1+\cdots +\ell_{k+1})}}{a_{\ell_1}\ldots a_{\ell_{k}}t_{\ell_{k+1}}}.
\end{eqnarray*}

\vspace{1mm}

On the other hand, if $\alpha$ is expanding with $\lim_{n\to \infty} t_{n}/t_{n+1}= \rho$ for $\rho>2$, we have that
\begin{eqnarray*}
\frac{\theta_\alpha(y)-\theta_\alpha(x)}{y-x}&\gg&\frac{2^{-(\ell_1+\cdots +\ell_{k})}}{a_{\ell_1}\ldots a_{\ell_{k}}}\cdot \frac{2^{-n}(1-2^{n-\ell_{k+1}})}{t_{n}}\\
&\gg& \frac{2^{-(\ell_1+\cdots +\ell_{k})}}{a_{\ell_1}\ldots a_{\ell_{k-1}}t_{\ell_{k}}}.
\end{eqnarray*}

\vspace{2mm}

\noindent{\bf Case 1.3} $\ $For the final part of the first case, suppose that
\[
[\ell_1, \ldots, \ell_{k+1}]_\alpha<y \leq [\ell_1, \ldots, \ell_{k+1}-1]_\alpha.
\]
In this situation, the argument used in Case 1.2 will no longer suffice. We must consider a further two subcases.

\vspace{2mm}

\noindent{\bf Subcase 1.3.1} $\ell_{k+2}>1$.

\vspace{1mm}

\noindent In the event that $\ell_{k+2}>1$, the point $[\ell_1, \ldots, \ell_{k+1}, 1]_\alpha$ still lies to the right of the point $x$. Then,
\begin{eqnarray*}
\frac{\theta_\alpha(y)-\theta_\alpha(x)}{y-x}&\geq&\frac{\theta_\alpha([\ell_1, \ldots, \ell_{k+1}]_\alpha)-\theta_\alpha([\ell_1, \ldots, \ell_{k+1}, 1]_\alpha)}{[\ell_1, \ldots, \ell_{k+1}-1]_\alpha-[\ell_1, \ldots, \ell_k]_\alpha}\\&=&\frac{2\cdot 2^{-(\ell_1+\cdots +\ell_{k+1})}(1-1/2)}{a_{\ell_1}\ldots a_{\ell_k}t_{\ell_{k+1}-1}}\\&\gg&
\frac{2\cdot2^{-(\ell_1+\cdots +\ell_{k+1})}}{a_{\ell_1}\ldots a_{\ell_k}t_{\ell_{k+1}}},
\end{eqnarray*}
where the last inequality again comes from the fact that $\alpha$ is expansive of exponent $\tau\geq 0$ or expanding.

\vspace{1mm}

\noindent{\bf Subcase 1.3.2} $\ell_{k+2}=1$.

\vspace{1mm}

\noindent In the event that $\ell_{k+2}>1$, the point $[\ell_1, \ldots, \ell_{k+1}, 1]_\alpha$ lies to the left of $x$ (it is equal to the $(k+2)$-th convergent). So, we make a slightly different calculation:
\begin{eqnarray*}
\frac{\theta_\alpha(y)-\theta_\alpha(x)}{y-x}&\geq&\frac{\theta_\alpha([\ell_1, \ldots, \ell_{k+1}]_\alpha)-\theta_\alpha([\ell_1, \ldots, \ell_{k+1}, 1, \ell_{k+3}]_\alpha)}{[\ell_1, \ldots, \ell_{k+1}-1]_\alpha-[\ell_1, \ldots, \ell_k, 1]_\alpha}\\&=&\frac{2\cdot 2^{-(\ell_1+\cdots +\ell_{k+1})}(2^{-1}-2^{-(1+\ell_{k+3})})}{a_{\ell_1}\ldots a_{\ell_k}t_{\ell_{k+1}}}\\&\gg&
\frac{2\cdot2^{-(\ell_1+\cdots +\ell_{k+1})}}{a_{\ell_1}\ldots a_{\ell_k}t_{\ell_{k+1}}}.
\end{eqnarray*}

This finishes all the permutations of the case where $\ell_{k+1}>1$. We now come to the case $\ell_{k+1}=1$. Again, this will be split into various cases. First notice that we can split up the interval $([\ell_1, \ldots, \ell_k, 1]_{\alpha}, [\ell_1, \ldots, \ell_{k-1}]_\alpha]$ using the points $[\ell_1, \ldots, \ell_k+n]_\alpha$ for $n\in\N$, as shown in Figure 3.3.

\begin{figure}[htbp]
\begin{center}
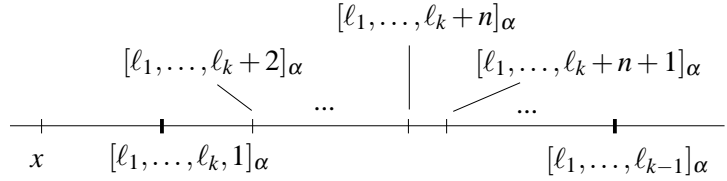\caption{Splitting up the interval $([\ell_1, \ldots, \ell_k, 1]_{\alpha}, [\ell_1, \ldots, \ell_{k-1}]_\alpha]$. }
\end{center}
\end{figure}

\noindent{\bf Case 2.1} Suppose that there exists $n\geq2$ such that
\[
[\ell_1, \ldots, \ell_k+n]_\alpha<y\leq[\ell_1, \ldots, \ell_k+n+1]_\alpha.
\]
Then,
\begin{eqnarray*}
\frac{\theta_\alpha(y)-\theta_\alpha(x)}{y-x}&\geq&\frac{\theta_\alpha([\ell_1, \ldots, \ell_k+n]_\alpha)-\theta_\alpha([\ell_1, \ldots, \ell_{k}, 1]_\alpha)}{[\ell_1, \ldots, \ell_{k}+n+1]_\alpha-[\ell_1, \ldots, \ell_k]_\alpha}\\
&=&\frac{2\cdot 2^{-(\ell_1+\cdots +\ell_{k})}(2^{-1}-2^{-n})}{a_{\ell_1}\ldots a_{\ell_{k-1}}(t_{\ell_{k}}-t_{\ell_{k}+n+1})}\\&\gg&
\frac{2\cdot2^{-(\ell_1+\cdots +\ell_{k})}}{a_{\ell_1}\ldots a_{\ell_{k-1}}t_{\ell_{k}}}.
\end{eqnarray*}

\vspace{2mm}

\noindent{\bf Case 2.2} Suppose that
\[
[\ell_1, \ldots, \ell_k, 1]_\alpha<y\leq[\ell_1, \ldots, \ell_k+2]_\alpha.
\]
We will again split this into two subcases.

\vspace{2mm}

\noindent{\bf Subcase 2.2.1} $\ell_{k+2}>1$.

\vspace{1mm}

\noindent In the event that $\ell_{k+2}>1$, the point $[\ell_1, \ldots,\ell_k,  1, 1]_\alpha$ lies to the right of the point $x$. Then,
\begin{eqnarray*}
\frac{\theta_\alpha(y)-\theta_\alpha(x)}{y-x}&\geq&\frac{\theta_\alpha([\ell_1, \ldots, \ell_k, 1]_\alpha)-\theta_\alpha([\ell_1, \ldots, \ell_k, 1, 1]_\alpha)}{[\ell_1, \ldots, \ell_k+2]_\alpha-[\ell_1, \ldots, \ell_k]_\alpha}\\&=&
\frac{2\cdot 2^{-(\ell_1+\cdots +\ell_{k})}(2^{-1}-2^{-1}+2^{-2})}{a_{\ell_1}\ldots a_{\ell_{k-1}}(t_{\ell_{k}}-t_{\ell_{k}+2})}\geq \frac{2\cdot2^{-(\ell_1+\cdots +\ell_{k})}}{a_{\ell_1}\ldots a_{\ell_{k-1}}t_{\ell_{k}}}.
\end{eqnarray*}

\vspace{1mm}

\noindent{\bf Subcase 2.2.2} $\ell_{k+2}=1$.

\vspace{1mm}

\noindent  We make a similar calculation as for Subcase 1.3.2.
\begin{eqnarray*}
\frac{\theta_\alpha(y)-\theta_\alpha(x)}{y-x}&\geq& \frac{\theta_\alpha([\ell_1, \ldots, \ell_k, 1]_\alpha)-\theta_\alpha([\ell_1, \ldots, \ell_k, 1, 1, \ell_{k+3}]_\alpha)}{[\ell_1, \ldots, \ell_k+2]_\alpha-[\ell_1, \ldots, \ell_k]_\alpha}\\&=&
\frac{2\cdot 2^{-(\ell_1+\cdots +\ell_{k})}(2^{-1}-2^{-1}+2^{-2}-2^{-(2+\ell_{k+3})})}{a_{\ell_1}\ldots a_{\ell_{k-1}}(t_{\ell_{k}}-t_{\ell_{k}+2})}\\
&\gg& \frac{2\cdot2^{-(\ell_1+\cdots +\ell_{k})}}{a_{\ell_1}\ldots a_{\ell_{k-1}}t_{\ell_{k}}}.
\end{eqnarray*}

This finishes Case 2. We have shown that for any $y>x$,
\[
\frac{\theta_\alpha(y)-\theta_\alpha(x)}{y-x}\gg \frac{2\cdot2^{-(\ell_1+\cdots +\ell_{k})}}{a_{\ell_1}\ldots a_{\ell_{k-1}}t_{\ell_{k}}}.
\]
A similar calculation can be done for $y<x$; we leave that to the reader. Thus the proof of the proposition is finished.
\end{proof}

\begin{rem}\label{limsupinf}
In \cite[Proposition 5.3~(i)]{mink?}, a similar result was proved for the Minkowski question mark function. However, the proof there contains a small mistake (the first inequality on page 2678 is incorrect) and is also incomplete (they do not consider the possibility that the  $(k+1)$-th continued fraction entry could equal one, in which case there are no intermediate approximants).
\end{rem}

The following corollary will be of use in the next section.

\begin{cor}\label{liminfderinfcor}
For each $x\in\U$, we have that
\[
\theta_\alpha'(x)=\infty\ \text{ if and only if }\ \lim_{k\to\infty}{\mu_\alpha\left(I^{(\alpha)}_k\right)}/{\lambda\Bigl(I^{(\alpha)}_k\Bigr)}=\infty.
\]
\end{cor}

\begin{proof}
If the derivative of $\theta_\alpha$ at $x$ exists in a generalised sense and $\theta_\alpha'(x)=\infty$, the conclusion of the corollary  follows directly from Lemma \ref{lemma1}.
For the other direction, recall that $[r_k^{(\alpha)}(x), r_{k+1}^{(\alpha)}(x)]_\pm=I_{\ell_1+\cdots \ell_{k+1}-1}^{(\alpha)}$ and so the sequence $\left(\mu_\alpha([r_k^{(\alpha)}, r_{k+1}^{(\alpha)}]_\pm)/\lambda([r_k^{(\alpha)}, r_{k+1}^{(\alpha)}]_\pm)\right)_{k\geq1}$ is a subsequence of the sequence $\left(\mu_\alpha(I^{(\alpha)}_k)/\lambda(I^{(\alpha)}_k)\right)_{k\geq1}$. Thus, the corollary is an immediate consequence of Proposition \ref{liminfderinf}.
\end{proof}

Let us now consider a condition which gives rise to points with derivative equal to zero (recall that almost every $x\in [0,1]$ is such that $\theta_\alpha'(x)=0$).

\begin{prop}\label{limitto0}
Suppose that $\alpha$ is either  {expansive of exponent $\tau>0$} or { expanding}. Let $x=[\ell_1, \ell_2, \ell_3, \ldots]_\alpha$ be such that
\[
\lim_{k\to\infty}\frac{\mu_\alpha\left(\left[r_k^{(\alpha)}(x), r_{k+1}^{(\alpha)}(x)\right]_\pm\right)}{\lambda\left(\left[r_k^{(\alpha)}(x), r_{k+1}^{(\alpha)}(x)\right]_\pm\right)}\cdot \frac{t_{\ell_{k+1}}}{a_{\ell_{k+1}}}=0.
\]
Then, $\theta_\alpha'(x)=0$.
\end{prop}

\begin{proof}
First, notice that
\[
\frac{\mu_\alpha\left(\left[r_k^{(\alpha)}(x), r_{k+1}^{(\alpha)}(x)\right]_\pm\right)}{\lambda\left(\left[r_k^{(\alpha)}(x), r_{k+1}^{(\alpha)}(x)\right]_\pm\right)}\cdot \frac{t_{\ell_{k+1}}}{a_{\ell_{k+1}}}=\frac{2\cdot 2^{-(\ell_1+\cdots +\ell_{k+1})}}{a_{\ell_1}\ldots a_{\ell_{k}}a_{\ell_{k+1}}}.
\]
The remainder of the proof consists of a series of simple calculations, as in the proof of Proposition \ref{liminfderinf}. We will make one case explicit and leave the rest to the reader. Let $y>x$. Then, for all $y$ close enough to $x$, there exists an even positive integer $k$ such that
$y\in ([\ell_1, \ldots, \ell_{k+1}]_\alpha, [\ell_1, \ldots, \ell_{k-1}]_\alpha]$. Suppose that $\ell_{k+1}>1$. As in Figure 3.2 in the proof of the previous proposition, we can locate $y$ with greater precision, as follows. First suppose that $[\ell_1, \ldots, \ell_{k}, 1]_\alpha<y\leq [\ell_1, \ldots, \ell_{k-1}]_\alpha$. Then we have that
\begin{eqnarray*}
\frac{\theta_\alpha(y)-\theta_\alpha(x)}{y-x}&\leq&\frac{\theta_\alpha([\ell_1, \ldots, \ell_{k-1}]_\alpha)-\theta_\alpha([\ell_1, \ldots, \ell_{k}]_\alpha)}{[\ell_1, \ldots, \ell_k, 1]_\alpha-[\ell_1, \ldots, \ell_k, \ell_{k+1}]_\alpha}\\
&=&  \frac{2\cdot2^{-(\ell_1+\cdots +\ell_k)}}{a_{\ell_1}\ldots a_{\ell_{k}}(1-t_{\ell_{k+1}})}\ll \frac{2\cdot2^{-(\ell_1+\cdots +\ell_k)}}{a_{\ell_1}\ldots a_{\ell_{k}}}.
\end{eqnarray*}
Now suppose that there exists a positive integer $n\in\{1, 2, \ldots, \ell_{k+1}-2\}$ such that $[\ell_1, \ldots, \ell_{k}, n+1]_\alpha<y\leq [\ell_1, \ldots, \ell_{k}, n]_\alpha$. In that case, we calculate
\begin{eqnarray*}
\frac{\theta_\alpha(y)-\theta_\alpha(x)}{y-x}&\leq&\frac{\theta_\alpha([\ell_1, \ldots, \ell_{k}, n]_\alpha)-\theta_\alpha([\ell_1, \ldots, \ell_{k}]_\alpha)}{[\ell_1, \ldots, \ell_k, n+1]_\alpha-[\ell_1, \ldots, \ell_k, \ell_{k+1}]_\alpha}\\
&=&  \frac{2\cdot2^{-(\ell_1+\cdots +\ell_k+n)}}{a_{\ell_1}\ldots a_{\ell_{k}}(t_{n+1}-t_{\ell_{k+1}})}\ll \frac{2\cdot2^{-(\ell_1+\cdots +\ell_k)}}{a_{\ell_1}\ldots a_{\ell_{k}}},
\end{eqnarray*}
where in this instance the final inequality holds in the case that $\alpha$ is expansive of exponent $\tau$ or $\alpha$ is expanding with $\lim_{n\to\infty}t_n/t_{n+1}=\rho$ and $1<\rho<2$. The case that $\alpha$ is expanding and $\rho>2$ must be considered separately, but the calculation is similar and we leave it to the reader.

Next, suppose that $[\ell_1, \ldots, \ell_{k+1}]_\alpha<y \leq [\ell_1, \ldots, \ell_{k+1}-1]_\alpha$ and $\ell_{k+2}>1$. In this case, we have that the point $[\ell_1, \ldots, \ell_{k+1},1 ]$ lies to the right of $x$ and we obtain that
\begin{eqnarray*}
\frac{\theta_\alpha(y)-\theta_\alpha(x)}{y-x}&\leq&\frac{\theta_\alpha([\ell_1, \ldots, \ell_{k+1}-1]_\alpha)-\theta_\alpha([\ell_1, \ldots, \ell_{k}]_\alpha)}{[\ell_1, \ldots, \ell_{k +1}]_\alpha-[\ell_1, \ldots, \ell_k, \ell_{k+1}, 1]_\alpha}\\
&\ll&  \frac{2\cdot2^{-(\ell_1+\cdots +\ell_{k+1})}}{a_{\ell_1}\ldots a_{\ell_{k+1}}}.
\end{eqnarray*}
Finally, if $[\ell_1, \ldots, \ell_{k+1}]_\alpha<y \leq [\ell_1, \ldots, \ell_{k+1}-1]_\alpha$ and $\ell_{k+2}=1$, we have that
\[
y-x\geq [\ell_1, \ldots, \ell_{k +1}]_\alpha-[\ell_1, \ldots, \ell_k, \ell_{k+1}, 1, \ell_{k+3}]_\alpha=a_{\ell_1}\ldots a_{\ell_{k+1}}(1-a_1t_{\ell_{k+3}})\gg a_{\ell_1}\ldots a_{\ell_{k+1}}.
\]
So, in this case too, we obtain that
\[
\frac{\theta_\alpha(y)-\theta_\alpha(x)}{y-x}\ll \frac{2\cdot2^{-(\ell_1+\cdots +\ell_{k+1})}}{a_{\ell_1}\ldots a_{\ell_{k+1}}}.
\]

To finish the proof, we must consider the case $\ell_{k+1}=1$ and also do similar calculations for points $y$ such that $x>y$. Both of these are similar to what we have done above, thus we leave the remaining details to the reader.
\end{proof}

\begin{rem}

Let us end this section with some remarks concerning the paper \cite{PVB}. In there, the authors consider first the function $\Phi_{2, \tau}$, which, although this is not made explicit, conjugates the tent system with the map $T_{\tau}$ which is given, for $\tau>1$, by
\[
T_\tau(x):=\left\{
             \begin{array}{ll}
               \tau x, & \hbox{for $x\in[0, 1/\tau)$;} \\
               \frac{\tau x -1}{\tau-1}, & \hbox{for $x\in [1/\tau, 1]$.}
             \end{array}
           \right.
\]
This is nothing other than an ``untwisted'' $\alpha$-Farey map, where ``untwisted'' means that the right-hand branch of the map has a positive slope. Let us denote such maps by $F_{\tilde{\alpha}}$. In this case, the partition in question, say $\widetilde{\alpha_\tau}$,  is given by $t_n:= \tau^{-(n-1)}$ and $a_n:=(\tau-1)/\tau^n$. Notice that this is simply a specific example of an expanding partition, since it certainly satisfies the condition $\lim_{n\to\infty}t_n/t_{n+1}=\rho>1$; in fact, here $\rho=\tau$.  The associated untwisted $\tilde{\alpha_\tau}$-L\"uroth map has all positive slopes. In this case the $\tilde{\alpha_\tau}$-L\"uroth coding is given by
\[
x=[\tilde{\ell}_1, \tilde{\ell}_2, \tilde{\ell}_3,\ldots]_{\tilde{\alpha_\tau}}=
t_{\tilde{\ell}_1+1}+a_{\tilde{\ell}_1}t_{\tilde{\ell}_2+1}+a_{\tilde{\ell}_1}a_{\tilde{\ell}_2}t_{\tilde{\ell}_3+1}+\cdots=
\frac{1}{\tau^{\tilde{\ell}_1}}+\frac{\tau-1}{\tau^{\tilde{\ell}_1+\tilde{\ell_2}}}+
\frac{(\tau-1)^2}{\tau^{\tilde{\ell}_1+\tilde{\ell_2}+\tilde{\ell}_3}}+\cdots
\]
The map equivalent to $\theta_\alpha$ in this positive slope situation is the map $\theta_{\tilde{\alpha}}$, which is defined by
\[
\theta_{\tilde{\alpha}}(x):=\sum_{k=1}^{\infty}2^{-\sum_{i=1}^k\tilde{\ell}_i(x)}.
\]
(For more details, we refer to \cite{sarathesis}.) The map $\Phi_{2, \tau}$ in the paper \cite{PVB} coincides with the inverse of the map $\theta_{\tilde{\alpha_\tau}}$. They first show that $\Phi_{2, \tau}$ is singular and then, assuming the derivative of $\Phi_{2, \tau}$ at a point $x$ exists in a generalised sense, give a condition in terms of a certain constant $K=K(\tau):=\frac{-\log(\tau-1)}{\log(2/\tau)}$ for which the derivative at the point $x$ is either equal to zero or is infinite. The proof boils down to an equivalent statement to Lemma \ref{lemma1}, which in their case states that if $\Phi_{2, \tau}'(x)$ exists it must satisfy
\[
\Phi_{2, \tau}'(x)=\lim_{n\to \infty} \frac{\lambda(C_{\tilde{\alpha_\tau}}(\tilde{\ell}_1, \ldots, \tilde{\ell}_n))}{2^{-\sum_{i=1}^n a_{\tilde{\ell}_i}}}=\lim_{n\to \infty} \frac{(\tau-1)^n\cdot 2^{\sum_{i=1}^n a_{\tilde{\ell}_i}}}{\tau^{{\sum_{i=1}^n {\tilde{\ell}_i}}}}=
\lim_{n\to \infty} \left(\left(\frac{2}{\tau}\right)^{{\sum_{i=1}^n {\tilde{\ell}_i}}/n}(\tau-1)\right)^n.
\]
Then the constant $K$ is just the boundary point between the term $\left(\frac{2}{\tau}\right)^{{\sum_{i=1}^n {\tilde{\ell}_i}}/n}(\tau-1)$ being strictly less than 1 or strictly greater than 1.

They then go on to generalise this by conjugating two expanding untwisted $\tilde{\alpha}$-Farey systems, one given by $\tilde{\alpha_\tau}$ with $t_n:=\tau^{-(n-1)}$ and the other given by $\tilde{\alpha_\beta}$ with $t_n:=\beta^{-(n-1)}$. They obtain a similar result for the map $\Phi_{\beta, \tau}$ which is the topological conjugacy map between the systems $F_{\tilde{\alpha_\beta}}$ and $F_{\tilde{\alpha_\tau}}$. Of course, $\Phi_{\beta, \tau}$ coincides with the composition $\theta_{\tilde{\alpha_\tau}}^{-1}\circ \theta_{\tilde{\alpha_\beta}}$. It may be interesting to consider conjugating homeomorphisms between two arbitrary $\alpha$-Farey maps, or even the case of two general expansive or expanding partitions (for either maps with positive or negative slopes).
\end{rem}

\section{Multifractal formalism for the $\alpha$-Farey system and the derivative of $F_\alpha$}

Let us now recall the outcome of the multifractal formalism for the $\alpha$-Farey system obtained in \cite{KMS}. Here, we must again assume that the partition $\alpha$ is either expanding or expansive of exponent $\tau\geq0$ and eventually decreasing (which means that for all sufficiently large $n$, we have that $a_n>a_{n+1}$), so this assumption will be made for every partition from here on.   For both the $\alpha$-L\"uroth and $\alpha$-Farey systems, the fractal-geometric description of the Lyapunov spectra were obtained by employing the general multifractal results of Jaerisch and Kesseb\"ohmer \cite{JaerischKess09}. First,  let the $\alpha$-Farey free-energy function $v:\R \to \R$ be defined by
\[
v(u):= \inf\left\{r\in\R: \sum_{n=1}^\infty a_n^u\exp(-rn)\leq 1\right\}.
\]
Let us also remind the reader that the Lyapunov exponent of a differentiable map $S:\U\to \U$ at a point $x\in\U$ is defined, provided the limit exists, by
\[
\Lambda(S, x):=\lim_{n\to \infty}\frac1n \sum_{k=0}^{n-1}\log|S'(S^k(x))|.
\]
The following result can be found in \cite{KMS}. (Here we have omitted the discussion of phase transitions and the boundary points of the spectrum, as they are not relevant to this paper.)

\noindent{\bf Theorem.} \cite[Theorem 3]{KMS} {\em Let $\alpha$ be either expanding or expansive of exponent $\tau\geq0$ and eventually decreasing. Then, where $s_-:=\inf\{-\log(a_n)/n:n\in\N\}$ and $s_+:=\sup\{-\log(a_n)/n:n\in\N\}$, we have that if $s\in (s_-, s_+)$, then
\[
\dim_{\mathrm{H}}(\{x\in \U:\Lambda(F_\alpha, x)=s\})=\inf_{u\in\R}\left\{u+s^{-1}v(u)\right\}.
\]}
We observe that it is equivalent to consider the free-energy function \[t(v):=\inf\left\{u\in\R: \sum_{n=1}^\infty a_n^u\exp(-nv)\leq 1\right\},\] in line with \cite{JaerischKess09}. The outcome then for the $\alpha$-Farey spectrum is that $\dim_{\mathrm{H}}(\{x\in \U:\Lambda(F_\alpha, x)=s\})=t^*(s):=\inf_{v\in\R}\{t(v)+vs^{-1}\}$.

In light of the results of the previous section, as already mentioned in the introduction, we can split the unit interval into three disjoint subsets, namely, $[0,1]=\Theta_0\cup \Theta_\infty\cup\Theta_\sim$. Recall that these sets are defined by $\Theta_0:=\{x\in \U: \theta_\alpha'(x)=0\}$, $\Theta_\infty:=\{x\in \U: \theta_\alpha'(x)=\infty\}$ and, finally, $\Theta_{\sim}:=\U\setminus \Theta_0\cup\Theta_\infty$. Observe that $\Theta_\sim$ can also be described as the set of points in $[0,1]$ at which the derivative of $\theta_\alpha$ does not exist. We already have that $\lambda(\Theta_0)=\dim_{\mathrm{H}}(\Theta_0)=1$. The aim of this section is to prove Theorem \ref{mainthm}, which describes the Hausdorff dimensions of the other two sets.
First, for $s\geq0$, recall the definition of the set $\mathcal{L}(s)$ from the introduction:
\[
\mathcal{L}(s):=\left\{x\in \U:\lim_{n\to\infty}\frac{\log(\lambda(I^{(\alpha)}_n(x)))}{-n}=s\right\}.
\]

Let us now prove the following useful lemma.

\begin{lem}\label{ctblsets}
For each $s\geq0$, we have that
\[
\dim_{\mathrm{H}}\left(\left\{x\in \U:\Lambda(F_\alpha, x)=s\right\}\right)=\dim_{\mathrm{H}}\left(\mathcal{L}(s)\right).
\]
\end{lem}

\begin{proof}
Firstly, from Proposition 4.2 in \cite{KMS}, where \[\Pi(L_{\alpha},x):=\lim_{n\to\infty}\frac{\sum_{k=1}^n\log(a_{\ell_{k}(x)})}{\sum_{k=1}^n\ell_k(x)},\] we have that the sets
\[ \left\{ x\in\U:\Pi(L_{\alpha},x)=s\right\} \hbox{ and }
\left\{ x\in\U:\Lambda(F_{\alpha},x)=s\right\} \]
 coincide up to a countable set of points. An almost identical argument (using \cite[Lemma 4.1~(1)]{KMS} as opposed to \cite[Lemma 4.1~(3)]{KMS}), shows that the same statement is true with the set
$\left\{ x\in\U:\Lambda(F_{\alpha},x)=s\right\} $ replaced by the set $\mathcal{L}(s)$. Combining these two statements yields the result.
\end{proof}

\begin{rem}
Notice that it follows immediately from Proposition \ref{ctblsets} that $\dim_{\mathrm{H}}(\mathcal{L}(s))=t^*(s)$.
\end{rem}


\begin{prop}\label{6.1}
\
\begin{itemize}
  \item[(a)] If $s\in(\log2, s_+]$, then
  \[
  \mathcal{L}(s)\subset \Theta_\infty.
  \]
  \item[(b)] If $s\in[s_-, \log2)$,  then
  \[
  \mathcal{L}(s)\subset \Theta_0.
  \]
  \item[(c)]
  \[
  \left\{x\in\U:\liminf_{n\to\infty} \frac{\sum_{i=1}^n\log(a_{\ell_{i}(x)})}{-\sum_{i=1}^n \ell_{i}(x)}<\log2<\limsup_{n\to\infty} \frac{\sum_{i=1}^n\log(a_{\ell_{i}(x)})}{-\sum_{i=1}^n \ell_{i}(x)}\right\}\subset \Theta_\sim.
  \]
\end{itemize}
\end{prop}

\begin{proof}
Let $x\in \mathcal{L}(s)$ be given. Then, for each $\varepsilon>0$, there exists $N_\varepsilon\in \N$ such that for all $n\geq N_\varepsilon$,
\[
n(s-\varepsilon)\leq \log\left(\frac1{\lambda(I_n^{(\alpha)}(x))}\right)\leq n(s+\varepsilon).
\]
In other words, recalling that $\mu_\alpha(I_n^{(\alpha)}(x))=2^{-n}$, we have that
\[
e^{-n(s+\varepsilon-\log2)}\leq \frac{\lambda(I_n^{(\alpha)}(x))}{\mu_\alpha(I_n^{(\alpha)}(x))}\leq e^{-n(s-\varepsilon-\log2)},
\]
for all $n\geq N_{\varepsilon}$. Thus, if $s\in (\log2, s_+]$, we deduce that
\[
\lim_{n\to\infty} \frac{\lambda(I_n^{(\alpha)}(x))}{\mu_\alpha(I_n^{(\alpha)}(x))}=0.
\]
By Corollary \ref{liminfderinfcor}, we then infer that $\theta_\alpha'(x)=\infty$ and so $x\in \Theta_\infty$. This proves part (a).

In order to prove part (b), first notice  (where the first equality can be proved similarly  to Lemma \ref{ctblsets} and the second comes from the proof of Lemma \ref{ctblsets}), that
\[
\lim_{n\to\infty} \frac{-\log(a_{\ell_1}\ldots a_{\ell_n}t_{\ell_{n+1}})}{\ell_1+\cdots +\ell_{n+1}}=\lim_{n\to\infty} \frac{-\log(a_{\ell_1}\ldots a_{\ell_{n+1}})}{\ell_1+\cdots +\ell_{n+1}}=\lim_{n\to\infty} \frac{\log(\lambda(I_n^{(\alpha)}(x)))}{-n}=s<\log2.
\]
Using this observation, a straightforward calculation along the lines of  that done for part (a) shows that we have $\lim_{n\to\infty}2^{-(\ell_1+\cdots +\ell_n)}/a_{\ell_1}\ldots a_{\ell_n}=0$. In light of Proposition \ref{limitto0}, we obtain that $\theta_\alpha'(x)=0$ and this finishes the proof of part (b).

Finally, to prove part (c), one immediately verifies that if $\liminf_{n\to\infty} \frac{\sum_{i=1}^n\log(a_{\ell_{i}(x)})}{-\sum_{i=1}^n \ell_{i}(x)}<\log2$, then there exists $0<c<1$ such that
\[
\liminf_{n\to\infty} \frac{a_{\ell_{1}(x)}\ldots a_{\ell_{n}(x)}}{2^{-\sum_{i=1}^n \ell_{i}(x)}}\leq e^c.
\]
Similarly, if $\limsup_{n\to\infty} {\sum_{i=1}^n\log(a_{\ell_{i}(x)})}/{(-\sum_{i=1}^n \ell_{i}(x))}>\log2$, then there exists $C>1$ such that
\[
\limsup_{n\to\infty} \frac{a_{\ell_{1}(x)}\ldots a_{\ell_{n}(x)}}{2^{-\sum_{i=1}^n \ell_{i}(x)}}\geq e^C.
\]
In other words, the limit as $n$ tends to infinity  of the sequence $\left(({a_{\ell_{1}(x)}\ldots a_{\ell_{n}(x)}})/(2^{-\sum_{i=1}^n \ell_{i}(x)})\right)_{n\geq1}$ does not exist. Therefore, the limit of the sequence $\left(\lambda(I_n^{(\alpha)}(x))/\mu_\alpha(I_n^{(\alpha)}(x))\right)_{n\geq1}$ does not exist either, and, in light of Lemma \ref{lemma1}, we have that the derivative $\theta_\alpha'(x)$ also cannot exist. This shows that $x\in\Theta_\sim$ and hence finishes the proof.
\end{proof}


For the next proposition, we define:
\begin{eqnarray*}
\mathcal{L}^{*}(s) & := & \left\{ x\in\mathcal{U}:\limsup_{n\rightarrow\infty}\frac{\log(a_{\ell_{1}(x)}\ldots a_{\ell_{n}(x)})}{{-\sum_{i=1}^n \ell_{i}(x)}}\geq s\right\} ,\\
\mathcal{L}_{*}(s) & := & \left\{ x\in\mathcal{U}:\liminf_{n\rightarrow\infty}\frac{\log(a_{\ell_{1}(x)}\ldots a_{\ell_{n}(x)})}{{-\sum_{i=1}^n \ell_{i}(x)}}\geq s\right\} ,\\
\mathcal{L}\left(s,t\right) & := & \left\{ x\in\mathcal{U}:\liminf_{n\to\infty}\frac{\log(a_{\ell_{1}(x)}\ldots a_{\ell_{n}(x)})}{{-\sum_{i=1}^n \ell_{i}(x)}}\leq s,\limsup_{n\to\infty}\frac{\log(a_{\ell_{1}(x)}\ldots a_{\ell_{n}(x)})}{{-\sum_{i=1}^n \ell_{i}(x)}}\geq t\right\} .\end{eqnarray*}

\begin{prop}\label{6.4}
\
\begin{itemize}
  \item [(a)] For each $s\in (s_-, s_+)$, we have that
\[
\dim_{\mathrm{H}}\left(\mathcal{L}_*(s)\right)=\dim_{\mathrm{H}}\left(\mathcal{L}^*(s)\right)=\dim_{\mathrm{H}}\left(\mathcal{L}(s)\right).
\]
  \item [(b)] For each $s_-<s_0\leq s_1<s_+$, we have that
\[
\dim_{\mathrm{H}}\left(\mathcal{L}(s_0, s_1)\right)=\dim_{\mathrm{H}}\left(\mathcal{L}(s_1)\right).
\]
\end{itemize}

\end{prop}

\begin{proof}
Towards part (a), the inequality $\dim_{\mathrm{H}}\left(\mathcal{L}_*(s)\right)\leq\dim_{\mathrm{H}}\left(\mathcal{L}^*(s)\right)$ is immediate from the fact that $\mathcal{L}_*(s)\subset\mathcal{L}^*(s)$. Also, notice that $\mathcal{L}(s)\subset\mathcal{L}_*(s)$, so we have the inequality $\dim_{\mathrm{H}}\left(\mathcal{L}(s)\right)\leq\dim_{\mathrm{H}}\left(\mathcal{L}_*(s)\right)$. To finish the proof of part (a), we will show, via a covering argument, that $\dim_{\mathrm{H}}\left(\mathcal{L}^*(s)\right)\leq t^*(s)$. For ease of exposition, let us define the two potential functions $\phi$ and $\psi$ by setting
\[
\phi(x):=\log(a_n)\ \text{ and }\ \psi(x):=-n, \ \text{ for }x\in A_n.
\]
Then,  for each $x\in \mathcal{L}^*(s)$ and every $\varepsilon>0$, we can choose $n_{k(x, \delta)}$ such that for all $k\geq k(x, \delta)$ we have that
\[
\diam(C_\alpha(\ell_1(x), \ldots, \ell_{n_k}(x)))=a_{\ell_1(x)}\ldots a_{\ell_{n_k}(x)}<\delta
\]
and
\[
0<\frac{S_{n_k}\psi(x)}{S_{n_k}\phi(x)}\leq \frac1s+\frac\varepsilon2,
\]
where the notation $S_n\phi$ denotes the $n$-th Birkhoff sum $\sum_{k=0}^{n-1}\phi\circ L_\alpha^k$. Thus, removing duplicates as necessary, we can cover the set $\mathcal{L}^*(s)$ with the family $\mathcal{A}_\delta$ of at most countably many cylinder sets, where
\[
\mathcal{A}_\delta:=\left\{C_i:=C_{\alpha}\left(\ell_1(x^{(i)}), \ldots, \ell_{n_{k(x^{(i)}, \delta)}}(x^{(i)})\right):i\in A\subseteq \N\right\}.
\]
Then, for all $\varepsilon>0$, where to shorten notation we have set $n_k:=n_{k(x, \delta)}$, we have that
\begin{eqnarray*}
\mathcal{H}_\delta^{t(v)+vs^{-1}+\varepsilon}\left(\mathcal{L}^*(s)\right)&\leq& \sum_{C_i\in \mathcal{A}_\delta}|C_i|^{t(v)+vs^{-1}+\varepsilon}\\
&=&\sum_{i\in A}\left(a_{\ell_1(x^{(i)})}\ldots a_{\ell_{n_k}(x^{(i)})}\right)^{t(v)+vs^{-1}+\varepsilon}\\
&=&\sum_{i\in A}\exp\left(S_{n_k}\phi(x^{(i)})(t(v)+vs^{-1}+\varepsilon)\right)\\
&\leq& \sum_{i\in A}\exp\left(S_{n_k}\phi(x^{(i)})\left(t(v)+v\frac{S_{n_k}\psi(x^{(i)})}{S_{n_k}\phi(x^{(i)})}+\frac{\varepsilon}{2}\right)\right)\\&\leq&
\sum_{n\in\N}\sum_{\ell_1, \ldots, \ell_n\in \N^n}\exp \sup_{y\in C_\alpha(\ell_1, \ldots, \ell_{n})}\left\{S_n\left(\left(t(v)+\frac\varepsilon2\right)\phi+v\psi\right)(y)\right\}.
\end{eqnarray*}
Recalling that the free-energy function $t$ is defined in terms of the pressure function $\mathcal{P}(t\phi+v\psi):=\log\sum_{n=1}^\infty a_n^t\exp(-vn)$ and that $\mathcal{P}$ is strictly decreasing as a function of $t$, from the definition of $t(v)$ it follows that $\mathcal{P}((t(v)+\varepsilon/2)\phi+v\psi)=\eta<0$. Consequently, for arbitrarily small $\delta$, we have that
\[
\mathcal{H}_\delta^{t(v)+vs^{-1}+\varepsilon}\left(\mathcal{L}^*(s)\right)\leq \sum_{n\in\N}e^{n\eta}<\infty,
\]
which is summable since $\eta<0$. Therefore, for every $\varepsilon>0$ and every $v\in \R$, we have that $\dim_{\mathrm{H}}\left(\mathcal{L}^*(s)\right)\leq t(v)+vs^{-1}+\varepsilon$. Finally, then,  we obtain  that \[
\dim_{\mathrm{H}}\left(\mathcal{L}^*(s)\right)\leq \dim_{\mathrm{H}}\left(\mathcal{L}(s)\right).
\]

Now, for the proof of part (b), first notice that since $\mathcal{L}(s_0, s_1)\subseteq\mathcal{L}^*(s_1)$ and $\dim_{\mathrm{H}}\left(\mathcal{L}^*(s_1)\right)=\dim_{\mathrm{H}}\left(\mathcal{L}(s_1)\right)$, it is clear that
\[
\dim_{\mathrm{H}}\left(\mathcal{L}(s_0, s_1)\right)\leq\dim_{\mathrm{H}}\left(\mathcal{L}(s_1)\right).
\]
To obtain the lower bound, where we denote by $C_n(x)$ the $n$-th level cylinder set containing the point $x$, it suffices to show (by, for instance, \cite[Proposition 2.3~(a)]{Fal2}), that there exists a finite measure $\mu$ such that
\begin{itemize}
\item[(i)] $\mu\left(\mathcal{L}\left(s_{0},s_{1}\right)\right)>0$,
\item[(ii)]  ${\displaystyle \liminf_{n\to\infty}\frac{-\log\mu\left(C_{n}(x)\right)}{S_{n}\phi(x)}\geq\dim_{H}\left(\mathcal{L}\left(s_{1}\right)\right)}$,
for all $x$ in a subset of $\mathcal{L}\left(s_{0},s_{1}\right)$ of positive $\mu$-measure.
\end{itemize}
In order to construct such a measure $\mu$, first note that it was shown in the proof of Theorem 3 in \cite{KMS} that for every $u<1$, there exists $v(u)$ such that \begin{eqnarray}\label{..}
\sum_{n=1}^\infty a_n^u \exp({-nv(u)})=1.
\end{eqnarray}
Therefore, for $s_0$ and $s_1$ we can find corresponding Bernoulli measures $\mathbb{P}_{s_0}$ and $\mathbb{P}_{s_1}$ which are defined by the probability vectors given by $p_n(s_0):=a_n^{u_{s_0}}\exp(-nv(u_{s_0}))$ and  $p_n(s_1):=a_n^{u_{s_1}}\exp(-nv(u_{s_1}))$, respectively.
Note that the relation between $u$ and $s$ is given by $-v'(u_{s_i})=s_i$, for $i=0,1$. It is then straightforward to show, by differentiating (\ref{..}) with respect to $u$, that $\int\phi\ \mathrm{d}\mathbb{P}_{s_i}/\int\psi\ \mathrm{d}\mathbb{P}_{s_i}=s_i$, again for $i=0,1$. We also have that for $\mathbb{P}_{s_i}$-a.e. $x\in \U$,
\[
\lim_{n\to\infty}\frac{1}{n}S_n\phi(x)=\int\phi\ \mathrm{d}\mathbb{P}_{s_i}\in(0, \infty)
\]
and
\[
\lim_{n\to\infty}\frac{-\log\mathbb{P}_{s_i}(C_n(x))}{S_n\phi(x)}=u_{s_i}+s_i^{-1}v(u_{s_i}).
\]
Therefore, by Egoroff's Theorem, there exists an increasing sequence of natural numbers $(m_k)_{k\geq1}$ and a sequence $(A_k)_{k\geq1}$ of Borel subsets of $\U$ such that $\mathbb{P}_{s_0}(A_{2k})\geq 1-2^{2k+1}$, $\mathbb{P}_{s_1}(A_{2k-1})\geq 1-2^{2k}$ and such that for all $x\in A_{2k}$ and all $n\geq m_{2k}$,
\begin{eqnarray}\label{eq4.2}
\left|\frac{1}{n}S_n\phi(x)-\int\phi\ \mathrm{d}\mathbb{P}_{s_0}\right|<\frac{1}{2k}\ \text{ and }\ \frac{-\log\mathbb{P}_{s_0}(C_n(x))}{S_n\phi(x)}>\dim_{\mathrm{H}}(\mathcal{L}(s_0))-\frac1{2k},
\end{eqnarray}
whereas for all $x\in A_{2k-1}$ and all $n\geq m_{2k-1}$,
\begin{eqnarray}\label{eq4.3}
\left|\frac{1}{n}S_n\phi(x)-\int\phi\ \mathrm{d}\mathbb{P}_{s_1}\right|<\frac{1}{2k-1}\ \text{ and }\ \frac{-\log\mathbb{P}_{s_1}(C_n(x))}{S_n\phi(x)}>\dim_{\mathrm{H}}(\mathcal{L}(s_1))-\frac1{2k-1}.
\end{eqnarray}

We now aim to use the sets $A_k$ to construct a set $\mathcal{M}\subset \mathcal{L}(s_0, s_1)$ by defining certain families of cylinder sets coded by increasingly long words and taking their intersection. To that end, set $n_0:=1+1/m_1$ and $n_k:=\prod_{i=1}^k(1+m_i)$, for each $k\geq1$. Then define the countable family of cylinder sets
\[
\mathcal{C}_k:=\{C_{n_{k-1}m_k}(x):x\in A_k\}, \ \text{ for each }k\geq1.
\]
Further define a second countable family of cylinder sets by setting $\mathcal{D}_1:=\mathcal{C}_1$ and setting
\[
\mathcal{D}_k:=\{DC:D\in \mathcal{D}_{k-1}, C\in \mathcal{C}_k\}, \ \text{ for each }k\geq2,
\]
where the cylinder set  $DC$ is obtained by concatenating the length $n_{k-1}$ word that defines $D$ and the length $n_{k-1}m_k$ word that defines $C$ and using this length $n_k$ word to define $DC$. Observe that if $x\in DC\in \mathcal{D}_k$, then $L_\alpha^{n_{k-1}}(x)\in C\in \mathcal{C}_k$. Finally, define
\[
\mathcal{M}:=\bigcap_{n\in\N}\bigcup_{I\in \mathcal{D}_k}I.
\]
Now, let $x\in \mathcal{D}_k$. Then,
\begin{eqnarray*}
\frac{S_{n_k}\phi(x)}{n_k}&=&\frac{S_{n_{k-1}}\phi(x)+S_{n_{k-1}m_k}\phi(L_\alpha^{n_{k-1}}(x))}{n_{k-1}(1+m_k)}\\
&=& \frac{1}{1+m_k}\cdot\frac{S_{n_{k-1}}\phi(x)}{n_{k-1}}+\frac{m_k}{1+m_k}\cdot\frac{S_{n_{k-1}m_k}\phi(L_\alpha^{n_{k-1}}(x))}{n_{k-1}m_k},
\end{eqnarray*}
and, since the latter equality is a convex combination, it follows immediately that the sequence $S_{n_k}\phi(x)/n_k$ is bounded. Therefore, where we have set $i(k):=k$ (mod 2), and recalling that $L_\alpha^{n_{k-1}}(x)\in A_k$,
\[
\lim_{k\to\infty}\left|\frac{S_{n_k}\phi(x)}{n_k}-\int\phi\ \mathrm{d}\mathbb{P}_{s_{i(k)}}\right|
=0
\]
This shows that for all $x\in \mathcal{M}$ we have two subsequences $(n_{2k})_{k\geq1}$ and $(n_{2k-1})_{k\geq1}$ along which we have that $\lim_{k\to\infty}S_{n_{2k}}\phi(x)/n_{2k}=\int\phi\ \mathrm{d}\mathbb{P}_{s_0}$ and $\lim_{k\to\infty}S_{n_{2k-1}}\phi(x)/n_{2k-1}=\int\phi\ \mathrm{d}\mathbb{P}_{s_1}$, which proves that $\mathcal{M}\in \mathcal{L}(s_0, s_1)$.

Now, using the Kolmogorov consistency theorem, define the probability measure $\mu$ on $\U$ by setting $\mu(C):=\mathbb{P}_{s_1}(C)$ for all length $n_1$ cylinder sets $C$ and, for all cylinder sets $I$ of the form $I=DC$, with $D$ of length $n_{k-1}$ and $C$ of length $n_{k-1}m_k$, setting $\mu(I):=\mu(D)\mathbb{P}_{s_{i(k)}}(C)$. Then, by construction,
\[
\mu(\mathcal{M})\geq \prod_{k\in \N}(1-2^{-k})>0.
\]
Thus, the measure $\mu$ satisfies condition (i).

To see that $\mu$ satisfies condition (ii), first note that every length $n_k$ cylinder set $C_{n_k}(x)$ for $x\in \mathcal{M}$ and $k\geq1$ can be split as follows: $C_{n_{k}}\left(x\right)=C_{n_{k-1}}\left(x\right)C_{m_{k}n_{k-1}}\left(L_\alpha^{n_{k-1}}(x)\right)$. Using this, we obtain that
\begin{eqnarray*}
\!\!\!\!\!\!\!\!\!\!\!\!\!\frac{-\log\left(\mu\left(C_{n_{k}}\left(x\right)\right)\right)}{S_{n_{k}}\phi\left(x\right)} & = & \frac{-\log\left(\mu\left(C_{n_{k-1}}\left(x\right)\right)\right)}{S_{n_{k-1}}\phi\left(x\right)}\cdot\frac{\frac{S_{n_{k-1}}\phi\left(x\right)}
{n_{k-1}}}{\frac{S_{n_{k}}\phi\left(x\right)}{n_{k}}}\cdot\frac{n_{k-1}}{n_{k}}\\
&  & +\frac{-\log\left(\mathbb{P}_{s_{i(k)}}\left(C_{m_{k}n_{k-1}}\left(L_\alpha^{n_{k-1}}(x)\right)\right)\right)}
{S_{m_{k}n_{k-1}}\phi\left(L_\alpha^{n_{k-1}}(x)\right)}\cdot{\frac{\frac{S_{m_{k}n_{k-1}}\phi\left(L_\alpha^{n_{k-1}}(x)\right)}{m_{k}n_{k-1}}}
{\frac{S_{n_{k}}\phi\left(x\right)}{n_{k}}}\frac{m_{k}n_{k-1}}{n_{k}}},\end{eqnarray*}
where the last ratio in the second term tends to 1 as $k$ tends to infinity.
This shows, similarly to the argument for condition (i), that since the above sum is a convex combination, the sequence $-\log\left(\mu\left(C_{n_{k}}\left(x\right)\right)\right)/S_{n_{k}}\phi\left(x\right)$ is also bounded. Therefore, given that $\dim_{\mathrm{H}}(\mathcal{L}(s_1))\leq \dim_{\mathrm{H}}(\mathcal{L}(s_0))$, we have that
\begin{equation}\label{eq6.3}
\liminf_{k\to\infty}\frac{-\log\left(\mu\left(C_{n_{k}}\left(x\right)\right)\right)}{S_{n_{k}}\phi\left(x\right)}\geq
\dim_{\mathrm{H}}(\mathcal{L}(s_1)).
\end{equation}
This shows that (ii) is satisfied along the subsequence $(n_k)_{k\geq1}$. To complete the proof, we must consider $n_k<n<n_{k+1}$. We will split this into two cases. Firstly, for $n_k<n<n_{k}+m_k$, one immediately verifies that
\[
\frac{-\log\left(\mu\left(C_{n}\left(x\right)\right)\right)}{S_{n}\phi\left(x\right)}\geq \frac{-\log\left(\mu\left(C_{n_k}\left(x\right)\right)\right)}{S_{n_k+m_k}\phi\left(x\right)}
=\frac{-\log\left(\mu\left(C_{n_{k}}\left(x\right)\right)\right)}
{S_{n_{k}}\phi(x)}{\cdot\frac{S_{n_{k}}\phi(x)/n_{k}}{S_{n_{k}+m_{k}}\phi(x)/(n_{k}+m_{k})}\cdot\frac{n_{k}}{n_{k}+m_{k}}},
\]
where again the last ratio on the right-hand side tends to 1 as $k$ (and therefore $n$) tends to infinity.
Secondly, if $n_{k}+m_{k}\leq n<n_{k+1}$ then $C_n(x)$ is equal to some length $n_k$ cylinder $D\in \mathcal{D}_k$ concatenated with the cylinder $C:=C_{n-n_k}(L_\alpha^{n_{k}})$, which has length at least equal to $m_k$. Since $x$ is assumed to belong to the set $\mathcal{M}$, the cylinder set $C$ contains some other cylinder set $I\in \mathcal{C}_{k+1}$. Thus,
\begin{eqnarray*}
\frac{-\log\left(\mu\left(C_{n}\left(x\right)\right)\right)}{S_{n}\phi\left(x\right)} &\geq& \frac{-\log\left(\mu\left(C_{n_{k}}\left(x\right)\right)\right)-
\log\mathbb{P}_{s_{i(k)}}\left(C_{n-n_k}\left(L_\alpha^{n_{k}}\left(x\right)\right)\right)}{S_{n}\phi\left(x\right)}\\
&\geq& \frac{-\log\left(\mu\left(C_{n_{k}}\left(x\right)\right)\right)}{S_{n_k}\phi(x)}\cdot\frac{S_{n_k}\phi(x)}{S_{n}\phi(x)}
+ \frac{-
\log\mathbb{P}_{s_{i(k)}}\left(C_{n-n_k}\left(L_\alpha^{n_{k}}\left(x\right)\right)\right)}{S_{n-n_k}\phi(L_\alpha^{n_{k}}\left(x\right))}
\cdot\frac{S_{n-n_k}\phi(L_\alpha^{n_{k}}\left(x\right))}{S_{n}\phi(x)}.
\end{eqnarray*}
Then, by (\ref{eq6.3}), for all $\varepsilon>0$ and all sufficiently large $k$ (and hence large $n$), we have that
\[
\frac{-\log\left(\mu\left(C_{n_{k}}\left(x\right)\right)\right)}{S_{n_k}\phi(x)}\geq \dim_{\mathrm{H}}(\mathcal{L}(s_1))-\varepsilon.
\]
Also, recalling that $n-n_k\geq m_k$, in light of (\ref{eq4.2}) and (\ref{eq4.3}), we obtain that
\[
\frac{-
\log\mathbb{P}_{s_{i(k)}}\left(C_{n-n_k}\left(L_\alpha^{n_{k}}\left(x\right)\right)\right)}{S_{n-n_k}\phi(L_\alpha^{n_{k}}\left(x\right))}
\geq \dim_{\mathrm{H}}(\mathcal{L}(s_{i(k)}))-\varepsilon\geq\dim_{\mathrm{H}}(\mathcal{L}(s_1))-\varepsilon.
\]
Finally, letting $\varepsilon$ tend to zero and combining (\ref{eq6.3}) with the calculations given  above for the two cases $n_k<n<n_{k}+m_k$ and $n_{k}+m_{k}\leq n<n_{k+1}$, we obtain that
\[
\liminf_{n\to\infty}\frac{-\log\left(\mu\left(C_{n}\left(x\right)\right)\right)}{S_{n}\phi(x)}\geq \dim_{\mathrm{H}}(\mathcal{L}(s_1)),
\]
which finishes the proof.
\end{proof}

\begin{rem}
The proof of the lower bound for Proposition \ref{6.4}~(b) follows along the same lines as the proof of \cite[Proposition 6.4]{mink?}, which in turn was inspired by the argument in \cite[Theorem 6.7(3)]{BS00}.
\end{rem}

We are now in a position to prove the main theorem.

\begin{proof}[Proof of Theorem \ref{mainthm}]
Firstly, that $\dim_{\mathrm{H}}(\mathcal{L}(\log2))<1$ follows immediately from the multifractal results in \cite[Theorem 3]{KMS}.

In order to prove that $\dim_{\mathrm{H}}\left(\Theta_\infty\right)=\dim_{\mathrm{H}}\left(\mathcal{L}\left(\log2\right)\right)$, it suffices to show that for every small enough $\delta>0$ we have
\[
\mathcal{L}(\log2 + \delta)\subset \Theta_\infty\subset \mathcal{L}_*(\log2).
\]
The first inclusion above is simply the statement of Proposition \ref{6.1}~(a). To demonstrate the second inclusion, let $x\in \Theta_\infty$ be given. Then, by Corollary \ref{liminfderinfcor}, we have that $\lim_{n\to\infty} 2^n\lambda(I^{(\alpha)}_n)=0$. Hence, for all $\varepsilon>0$ there exists $n_\varepsilon\in\N$ such that for all $n\geq n_\varepsilon$ we have that
\begin{eqnarray*}
2^n\lambda(I^{(\alpha)}_n)<\varepsilon& \Rightarrow & \log\left(\lambda(I^{(\alpha)}_n)\right)<-n\log2 +\log\varepsilon\\
& \Rightarrow &\frac{\log\left(\lambda(I^{(\alpha)}_n)\right)}{-n}>\log2 -\frac{\log\varepsilon}n.
\end{eqnarray*}
Therefore it follows that
\[
\liminf_{n\to\infty}\frac{S_n\phi(x)}{S_n\psi(x)}\geq \liminf_{n\to\infty}\frac{\log\left(\lambda(I^{(\alpha)}_n)\right)}{-n}\geq\log2,
\]
which shows that $x\in\mathcal{L}_*(\log2)$. Consequently, $\Theta_\infty\subset\mathcal{L}_*(\log2)$, as required.

To prove that $\dim_{\mathrm{H}}\left(\Theta_\sim\right)\leq\dim_{\mathrm{H}}\left(\mathcal{L}\left(\log2\right)\right)$, by Proposition \ref{6.4}~(a), it is enough to show that $\Theta_\sim\subset \mathcal{L}^*\left(\log2\right)$. Towards this end, let $x\in \Theta_\sim$. Hence $x\in \U\setminus \Theta_0$ and, according to Proposition \ref{limitto0}, we have that
\[
\limsup_{k\to\infty}\frac{\mu_\alpha\left(\left[r_k^{(\alpha)}(x), r_{k+1}^{(\alpha)}(x)\right]\right)}{\lambda\left(\left[r_k^{(\alpha)}(x), r_{k+1}^{(\alpha)}(x)\right]\right)}\cdot \frac{t_{\ell_{k+1}(x)}}{a_{\ell_{k+1}(x)}}=\limsup_{k\to\infty}\frac{2^{-(\ell_1(x)+\cdots +\ell_{k}(x))}}{a_{\ell_1(x)}\ldots a_{\ell_k(x)}}>0\Rightarrow\limsup_{n\to\infty}\frac{S_n\phi(x)}{S_n\psi(x)}\geq \log2.
\]
This implies that $x\in \mathcal{L}^*(\log2)$ and so $\Theta_\sim\subset \mathcal{L}^*\left(\log2\right)$.

For the lower bound, $\dim_{\mathrm{H}}\left(\Theta_\sim\right)\geq\dim_{\mathrm{H}}\left(\mathcal{L}\left(\log2\right)\right)$, recall that in Proposition \ref{6.1}~(c) we proved that
\[
  \left\{x\in\U:\liminf_{n\to\infty} \frac{\sum_{i=1}^n\log(a_{\ell_{i}(x)})}{-\sum_{i=1}^n \ell_{i}(x)}<\log2<\limsup_{n\to\infty} \frac{\sum_{i=1}^n\log(a_{\ell_{i}(x)})}{-\sum_{i=1}^n \ell_{i}(x)}\right\}\subset \Theta_\sim.
  \]
Then, due to Proposition \ref{6.4}~(b), we have that   $\dim_{\mathrm{H}}\left(\Theta_\sim\right)\geq\dim_{\mathrm{H}}\left(\mathcal{L}(s_1)\right)$ for all $s_1\in(\log2, s_+)$. This finishes the proof.
\end{proof}

\end{document}